\documentclass{amsart}
\usepackage{amssymb}
\usepackage{amsfonts}
\usepackage[all]{xy}
\usepackage{amssymb}
\usepackage{amsmath}
\usepackage{amsthm}
\usepackage{enumerate}
\usepackage{tabularx}
\usepackage{pdfpages}
\usepackage{centernot}
\usepackage{mathtools}
\usepackage{amsthm,amssymb}
\usepackage{etoolbox}
\AtBeginEnvironment{proof}{\setlength{\parindent}{0pt}}
\usepackage{url}
\usepackage{tikz}
\usepackage{amssymb}
\usetikzlibrary{matrix}
\usepackage{tikz-cd}
\usepackage{tikz}
\usepackage{marginnote}
\definecolor{mygray}{gray}{0.85}
\usepackage[backgroundcolor=mygray,colorinlistoftodos,prependcaption,textsize=small]{todonotes}
\usepackage{xargs}

\newcommand{\mrm}[1]{\mathrm{#1}}

\renewcommand{\leq}{\leqslant}
\renewcommand{\geq}{\geqslant}

\newcommand{\cupdot}{\mathbin{\dot{\cup}}}

\makeatletter
\def\subsection{\@startsection{subsection}{3}%
  \z@{.5\linespacing\@plus.7\linespacing}{.3\linespacing}%
  {\bfseries\centering}}
\makeatother

\makeatletter
\def\subsubsection{\@startsection{subsubsection}{3}%
  \z@{.5\linespacing\@plus.7\linespacing}{.3\linespacing}%
  {\centering}}
\makeatother

\makeatletter
\def\myfnt{\ifx\protect\@typeset@protect\expandafter\footnote\else\expandafter\@gobble\fi}
\makeatother

\newcommand{\cC}{\mathcal{C}}

\newcommand{\Op}{\mathcal{O}}
\newcommand{\isom}{\cong}
\newcommand{\tensor}{\otimes}

\newcommand{\dis}{\displaystyle}

\renewcommand{\restriction}{ {\upharpoonright} }
\newcommand*{\tp}{\mathrm{tp}}
\newcommand*{\End}{\mathrm{End}}
\newcommand*{\Cl}{\mathrm{Cl}}
\newcommand*{\Pun}{\mathrm{Pun}}

\newcommand*{\op}{\mathrm{op}}
\newcommand*{\Th}{\mathrm{Th}}
\newcommand*{\Zg}{\mathrm{Zg}}
\newcommand*{\gra}{\alpha}
\newcommand*{\grb}{\beta}
\newcommand*{\grd}{\delta}
\newcommand*{\grf}{\varphi}
\newcommand*{\grg}{\gamma}

\newcommand*{\grD}{\Delta}
\newcommand*{\grl}{\lambda}

\newcommand*{\gro}{\omega}

\newcommand*{\bfa}{{\bf a}}
\newcommand*{\bfb}{{\bf b}}
\newcommand*{\bfu}{{\bf u}}
\newcommand*{\bfv}{{\bf v}}
\newcommand*{\bfw}{{\bf w}}

\newcommand*{\bfzero}{{\bf 0}}
\newcommand*{\RMod}{R\mbox{-}\mathrm{Mod}}
\newcommand*{\RProj}{R\mbox{-}\mathrm{Proj}}
\newcommand*{\Ab}{\mathrm{Ab}}

\newtheorem{theorem}{Theorem}[section]

\usepackage[colorlinks,citecolor=blue,urlcolor=blue, linkcolor=blue]{hyperref}
\newtheorem{corollary}[theorem]{Corollary}
\newtheorem{lemma}[theorem]{Lemma}
\newtheorem{proposition}[theorem]{Proposition}
\newtheorem{example}[theorem]{Example}

\newtheorem{fact}[theorem]{Fact}
\newtheorem*{repeatedresult}{\repeatedresultheading}
\newenvironment{restatement}[2]
  {\def\repeatedresultheading{#1~\ref{#2}}\begin{repeatedresult}}
  {\end{repeatedresult}}

\theoremstyle{definition}

\newtheorem{definition}[theorem]{Definition}
\newtheorem{remark}[theorem]{Remark}

\newtheorem{observation}[theorem]{Observation}
\newtheorem{notation}[theorem]{Notation}

\usepackage[backgroundcolor=mygray,colorinlistoftodos,prependcaption,textsize=small]{todonotes}
\usepackage{xcolor}

\newcommand{\pureindep}[1][]{%
  \mathrel{
    \mathop{
      \vcenter{
        \hbox{\oalign{\noalign{\kern-.3ex}\hfil$\vert$\hfil\cr
              \noalign{\kern-.7ex}
              $\smile$\cr\noalign{\kern-.3ex}}}
      }
    }\displaylimits_{#1}
  }
}

\numberwithin{equation}{section}

\begin{document}

\begin{abstract}
Let $R$ be a countable ring. Given an $R$-module $A$, we call a decomposition
$A = \bigoplus_{i \in I} N_i$ a \emph{Kaplansky decomposition} if each $N_i$ is
countable. We characterize the uncountable Polish $R$-modules that admit a
Kaplansky decomposition: they are exactly the modules of the form
$B \oplus M^{\omega}$, where $B$ and $M$ are countable and $M$ is
$\Sigma$-algebraically compact. The countable summand $B$ may moreover be taken
to be an elementary submodule satisfying a closure condition, which makes
$M^{\omega}$ unique up to isomorphism, and hence an invariant of $A$.
We use this to characterize the countable rings admitting a free uncountable
Polish $R$-module, generalizing results of Shelah and Solecki. This class of
rings has a purely ring-theoretic description: it consists exactly of the
countable left perfect and right coherent rings, i.e.\ the rings identified by
Chase's theorem on products of projective modules. We observe that these are
also exactly the countable $F$-rings, i.e.\ those countable rings $R$ for which $R^{\omega}$ is
free. Finally, we give a ring-theoretic characterization of the countable rings
$R$ for which there exists an uncountable projective Polish $R$-module.
\end{abstract}

\title{Kaplansky decompositions of Polish modules}

\thanks{No. 1279 on Shelah's publication list. Research of G. Paolini was  supported by project PRIN 2022 ``Models, sets and classifications", prot. 2022TECZJA, and by INdAM Project 2024 (Consolidator grant) ``Groups, Crystals and Classifications''. INdAM Project 2024 also supported a summer visit connected with this project. Research of S. Shelah was supported by SF 2320/23: The Israel Science Foundation (ISF) (2023-2027).}

\author{Ivo Herzog}
\address{Department of Mathematics, The Ohio State University at Lima, 4240 Campus Drive, Lima, OH 45804, USA.}
\email{herzog.23@osu.edu}

\author{Gianluca Paolini}
\address{Department of Mathematics ``Giuseppe Peano'', University of Torino, Via Carlo Alberto 10, 10123 Torino, Italy.}
\email{gianluca.paolini@unito.it}

\author{Saharon Shelah}
\address{Einstein Institute of Mathematics, The Hebrew University of Jerusalem, Edmond J. Safra Campus, Givat Ram, Jerusalem 9190401, Israel \and Department of Mathematics, Rutgers University, Hill Center, 110 Frelinghuysen Road, Piscataway, NJ 08854-8019, USA.}
\email{shelah@math.huji.ac.il}

\makeatletter
\patchcmd{\@maketitle}{\global\topskip42\p@\relax}
  {\global\topskip24\p@\relax}{}{}
\patchcmd{\@setauthors}{\@topsep30\p@\relax}
  {\@topsep20\p@\relax}{}{}
\patchcmd{\@maketitle}{\dimen@34\p@}
  {\dimen@22\p@}{}{}
\makeatother
\date{\today}
\maketitle

\setcounter{tocdepth}{2}
\vspace{-1.5em}
\tableofcontents

\newpage 
\section{Introduction}
\begingroup
\setlength{\parskip}{\smallskipamount}

The question of whether algebraically free objects can carry Polish group
topologies arises in descriptive set theory. In 1997, D.\ Evans asked whether an
uncountable free group could be non-Archimedean Polish~\cite[Introduction, footnote~1]{1121}, while H.\ Becker and
A.\ S.\ Kechris asked the corresponding question without the
non-Archimedean assumption~\cite{BeckerKechris}; see also~\cite[Introduction]{Shelah2011}. Shelah answered these questions negatively
\cite{Shelah2003,Shelah2011}, and Solecki proved that no uncountable free
abelian group admits a Polish group topology \cite{Solecki}. Both nonexistence results also follow from Dudley's earlier automatic
continuity theorem \cite{Dudley}, which forces any completely metrizable group
topology on a free or free abelian group to be discrete. These consequences in the Polish setting were already noted by Rosendal~\cite[pp.~198--199, Theorem~3.2 and Corollary~3.3]{Rosendal}. These results form part of a line of work in
which topological and descriptive-set-theoretic methods impose strong
restrictions on the underlying algebraic structure of a Polish group or module~\cite{PS17,1121,Rosendal}. More recently, Bogopolski and Corson~\cite{BogopolskiCorson} proved related impossibility results for acylindrically hyperbolic groups; in particular, such groups admit no nondiscrete completely metrizable group topology.

Recall that a Polish $R$-module is an $R$-module endowed with a separable completely metrizable topology for which addition and scalar multiplication are continuous. It is non-Archimedean if the underlying topological group has a countable neighborhood basis at the identity consisting of open subgroups. We write $\mathfrak{c}=2^{\aleph_0}$ for the cardinality of the continuum.

A largely separate line of research in module theory studied the behavior of
direct products and direct-sum decompositions. Products of modules rarely
decompose as direct sums, and one may ask more generally when an additive
subcategory of modules is closed under products. Beginning with the work of
Chase~\cite{chase,chase1} and Zimmermann-Huisgen~\cite{ZH80}, and continuing
in, among others, \cite{kra_sao,okoh,ON84,oneill}, this program led to structural
criteria involving coherence, perfectness, and the theory of $F$- and
$P$-rings.
In this terminology, an $F$-ring is a ring $R$ for which the $R$-module $({}_{R}R)^\omega$ is free, while a $P$-ring is one for which $({}_{R}R)^\omega$ is projective; see also~\cite{chase,chase1,lenzing,ON84,oneill} on this topic.

 The present paper brings these two lines of research together. The bridge is
that a countable product of countable discrete modules carries a natural non-Archimedean Polish $R$-module topology, whereas a decomposition of a Polish module into countable summands
can be constrained by Baire-category and descriptive-set-theoretic arguments.
Positive-primitive formulae and the model theory of modules translate these
topological constraints into algebraic information about decompositions. The main contribution of the present paper is to carry out this translation: we characterize Polish modules admitting a direct-sum decomposition into countable summands and resolve the existence questions for free and projective Polish modules over countable rings. In particular, we extend the free-abelian problem settled by Solecki~\cite{Solecki} to free modules over arbitrary countable rings.

 Soon after Kaplansky~\cite{Kap} proved that every projective $R$-module is a direct sum of countably generated ones, Chase~\cite[Theorem~3.3]{chase} characterized the rings for which arbitrary products of projective modules are projective: they are exactly the right coherent, left perfect rings. These conditions have useful model-theoretic interpretations. Right coherence means that the pp-definable subgroups of ${}_{R}R$ in one variable are precisely its finitely generated right ideals~\cite[Theorem~14.16]{prest_first_book}. Under this hypothesis, the Bass--Bj\"ork minimum condition on finitely generated right ideals identifies left perfectness with $\Sigma$-algebraic compactness of ${}_{R}R$~\cite{Bass,Bjork}.

 The purpose of this paper is to establish that analogous properties hold for an uncountable Polish module ${_R}A$ over a countable ring $R$, assuming there exists a direct-sum decomposition ${_R}A = \bigoplus_{i \in I} \; A_i$ into countable summands; we call such a decomposition a \textbf{Kaplansky decomposition}\footnote{Throughout, the summands in indexed direct sums are assumed to be nonzero.}.
 
Informally, our main result says that the uncountable Polish modules with a Kaplansky decomposition are exactly those of the form $B\oplus M^\omega$, where $B$ and $M$ are countable and $M$ is nonzero and $\Sigma$-algebraically compact. The last condition means that descending chains of pp-definable subgroups stabilize. Thus, up to a countable direct summand, these modules have the algebraic form of a countable product of one countable module. The theorem concerns the underlying module; its given Polish topology need not be a product topology.

Our first main result is the following descending chain condition (DCC) criterion.

\begin{theorem}\label{Thm_DCC}
	Let ${_R}A =  \bigoplus_{i \in I} \; A_i$ be a Kaplansky decomposition of an uncountable Polish module over a countable ring $R$. There is a partition of the index set $I = I_1 \cupdot I_2$ with $I_1$ countable such that the uncountable direct summand $U$ of 
\begin{equation} \label{decomp 1}
A = (\bigoplus_{i \in I_1} \; A_i) \; \oplus \; (\bigoplus_{i \in I_2} \; A_i) = B \oplus U
\end{equation}
is $\Sigma$-algebraically compact. 
\end{theorem}

Equivalently, after removing countably many of the given summands, every descending chain of pp-definable subgroups of $U$ stabilizes. No closedness of the summands or continuity of the coordinate projections is assumed.

This allows us to invoke Garavaglia's decomposition theory of $\Sigma$-algebraically compact modules~\cite{garavaglia} and refine the decomposition of the uncountable summand $U = \bigoplus_{i \in I_2} \; A_i$ into a Krull--Schmidt decomposition, i.e., as a direct sum of indecomposable modules with local endomorphism rings. Such a decomposition is essentially unique~\cite[\S V.5]{stenstrom} and, after collecting isomorphic indecomposable summands into homogeneous components, it can be expressed as a direct sum of the form:
\begin{equation} \label{decomp 2}
U = \bigoplus_{j \in J} \;\; U_j^{(\lambda_j)},
\end{equation}
where each $U_j$ is indecomposable with a local endomorphism ring (and $i \neq j$ implies $U_i \not\cong U_j$).  

  Although the decomposition from Theorem~\ref{Thm_DCC} is not unique up to isomorphism, we can extract an important invariant using some basic model theory of modules. Recall that if $\vdash \psi (\bfu) \to \varphi (\bfu)$ is an implication of pp-formulae in the language of $R$-modules, then the quotient $\varphi(M)/\psi(M)$ has a cardinal index in any module ${_R}M$. The Baur--Monk invariants record its exact value when this index is finite and record only that it is infinite otherwise; the collection of these finite-or-infinite invariants determines the complete theory $T=\Th(M)$.

An important feature (Corollary~\ref{Polish_invariants}) of a Polish $R$-module $A$ is that if the index $[\varphi (A) \colon \psi(A)]$ is uncountable, then it has the cardinality of the continuum. We may therefore associate with it the \textbf{Polish unlimited theory of} $A$, given by the axioms 
\begin{equation}\Th_{\Pun} (A) \vdash [\varphi \; \colon \psi] = 
\left\{ \begin{array}{ll} \infty & \mbox{if} \;\; [\grf (A) \colon \psi (A)] = \mathfrak{c} ;\\
1 & \mbox{if} \;\; [\grf (A) \colon \psi (A)] \leq \aleph_0 ,\end{array} \right.
\end{equation}
which is consistent: if $B\preccurlyeq A$ is a bin as in Definition~\ref{the bin}, then Proposition~\ref{bin_quotient} shows that $A/B\models\Th_{\Pun}(A)$. Theorem~\ref{Thm_DCC} therefore implies that the theory $\Th_{\Pun} (A)$ is totally transcendental, when the Polish module $A$ admits a Kaplansky decomposition. 

Another purpose of the paper is to place the basic ideas from the model theory of modules into the context of descriptive set theory. Quantifier-free pp-formulae define closed subgroups of finite powers of a Polish module. More generally, every pp-definable subgroup is Borel and carries a Polish group topology obtained as a quotient of a closed subgroup.
\noindent This topology may be finer than the inherited topology, but it gives exactly the same analytic subsets; see Lemma~\ref{pp_polish}. This allows the meager/nonmeager dichotomy of Baire-category arguments to play itself out in several ways (Theorems~\ref{Thm_Second_qf} and~\ref{analytic_dichotomy}) in a section of the paper devoted to the model theory of Polish modules. The theorems of Pettis~\cite[9.9]{K1} and Mycielski~\cite[19.1]{K1} are applied to establish a remarkable saturation property of Polish modules, which controls pp-indices arising from countable ascending unions in the following way: if a pp-type that is given by a single pp-formula and omits an ascending union of pp-subformulae remains consistent with the Polish unlimited theory, then not only the set of its realizations has cardinality $\mathfrak c$, but the remaining quotient must have continuum many cosets. This is the content of the next theorem.

\begin{theorem} \label{Thm_CH}
	Let $R$ be a countable ring, ${_R}A$ an uncountable Polish module and consider an ascending chain of pp-definable subgroups
	\begin{equation}\psi_0 (A) \subseteq \psi_1 (A) \subseteq \cdots \subseteq \psi_i (A) \subseteq \cdots \subseteq \varphi (A) \subseteq A^n\end{equation}
	that is bounded above by the pp-definable subgroup $\varphi (A)$. If the pp-type 
	\begin{equation}q(\bfu) = \{ \varphi (\bfu) \} \cup \{ \neg \psi_i (\bfu) \colon i < \omega \}\end{equation}
	is $\Th_{{\rm Pun}} (A)$-consistent, then the index 
	$[\varphi (A) \; \colon \Sigma_i \; \psi_i (A)]$ has the cardinality of the continuum $\mathfrak{c}$.	
\end{theorem}

The decomposition of Theorem~\ref{Thm_DCC} can be adjusted so that $A = B \oplus U$, where $B$ is a countable elementary summand with a closure condition, i.e., a {\em bin,} and $U \models \Th_{\Pun}(A)$. Theorem~\ref{Thm_CH} can then be used to show that the isomorphism types of the indecomposable summands $U_j$ in Decomposition~(\ref{decomp 2}) are precisely the points of the totally transcendental closed subset $\Cl (\Th_{\Pun} (A)) \subseteq \Zg (R)$ of the Ziegler spectrum associated with the theory $\Th_{\Pun}(A)$, and that each occurs with multiplicity $\grl_j=\mathfrak{c}$. Thus this closed set records exactly the indecomposable summands occurring in $U$, and the unlimited summand $U$ is unique up to isomorphism. In order to state our characterization of Polish modules with a Kaplansky decomposition, let us associate to any totally transcendental closed subset $\cC \subseteq \Zg (R)$ of the Ziegler spectrum the countable module $M_{\cC} := \bigoplus \, \{ U \colon U \in \cC \}$ with one copy of $U$ chosen for each point of $\cC$. The closed set $\cC \subseteq \Zg (R)$ is totally transcendental if and only if $M_{\cC}$ is $\Sigma$-algebraically compact.

\begin{theorem}[Main Theorem]\label{main_th} 
	Let $A$ be an $R$-module with a Kaplansky decomposition $A = \bigoplus_{i \in I} A_i$. Then $A$ admits a Polish $R$-module structure if and only if there is a partition $I = I_1 \cupdot I_2$ and a totally transcendental closed set $\cC \subseteq \Zg (R)$ such that $A = B \oplus U$, where
	\begin{enumerate}[(1)]
		\item $B = \bigoplus_{i \in I_1} A_i$ is a bin (equivalently, a countable elementary submodule such that $A/B \models \Th_{\Pun}(A)$, see Definition~\ref{the bin}) of $A$ (so that $I_1$ is necessarily countable); and 
		\item $U = \bigoplus_{i \in I_2} A_i \cong (M_{\cC})^{(\mathfrak{c})}$.
	\end{enumerate}
	Furthermore, $\cC = \Cl (\Th_{\Pun} (A))$ and $U \cong {(M_{\cC})}^{\gro}$ are invariants associated to $A$.
\end{theorem}

Notice  that the topology obtained from the backward implication of the Main Theorem is non-Archimedean: the countable module $B$ is given the discrete topology, while $U \cong (M_{\cC})^{\gro}$ carries the product topology. Consequently, an $R$-module with a Kaplansky decomposition admits a Polish $R$-module topology if and only if it admits a non-Archimedean one. This is an immediate \mbox{corollary of Theorem~\ref{main_th}.}

 In Section~\ref{sec_applications} we apply the Main Theorem to the existence of uncountable free Polish modules, returning to the questions about free groups and free abelian groups studied by Shelah and Solecki~\cite{Shelah2003,Shelah2011,Solecki}. For an uncountable free Polish module, removing countably many free summands leaves an isomorphic module, so the DCC criterion (Theorem~\ref{Thm_DCC}) forces ${}_{R}R$ to be $\Sigma$-algebraically compact. A further application of Theorem~\ref{Thm_CH} gives right coherence. Conversely, Chase's theorem and O'Neill's freeness criterion show that, over a countable right coherent, left perfect ring, $R^\omega\cong R^{(\mathfrak{c})}$. We obtain the following characterization, whose equivalence (2)$\Leftrightarrow$(3) is Chase's theorem.

\begin{theorem} \label{Free_Polish} The following are equivalent for a countable ring $R$:
\begin{enumerate}[(1)]
	\item there exists an uncountable free $R$-module with a Polish topology;
	\item the ring $R$ is right coherent and left perfect;
	\item the class $\RProj$ of projective $R$-modules is closed under products; and
	
	\item $({}_{R}R)^{\gro} \isom ({}_{R}R)^{(\mathfrak{c})}$ is a free $R$-module.
	
\end{enumerate}  
\end{theorem}

Thus the rings in Theorem~\ref{Free_Polish} are exactly the countable $F$-rings. We also observe that the $F$-ring and $P$-ring conditions, studied by O'Neill~\cite{ON84,oneill}, coincide for countable rings; see Corollary~\ref{F_theorem}.

To illustrate Theorem~\ref{Free_Polish}, the free $\mathbb Z$-module $\mathbb Z^{(\mathfrak{c})}$ admits no Polish group topology, by results of Dudley~\cite{Dudley} and Solecki~\cite{Solecki}. In contrast, for every countable field $K$, the uncountable free $K$-module $K^{(\mathfrak{c})}\cong K^\omega$ admits a Polish module topology, transported from the product of countably many copies of the discrete field $K$. When $K$ is finite, this topology is compact.

For projective modules, we obtain the following characterization. Condition~(4) of Theorem~\ref{Polish projective existence} gives a direct construction by tensoring a countable power of the corner ring $eRe$ with $Re$. The converse is the main point: the Main Theorem supplies an indecomposable endofinite projective summand, from which we obtain the artinian corner in (3). An idempotent is called \emph{irreducible} (or \emph{primitive}) if it is nonzero and cannot be written as a sum of two nonzero orthogonal idempotents, where idempotents $e$ and $f$ are orthogonal if $ef=fe=0$.

\begin{theorem} \label{Polish projective existence}
	The following are equivalent for a countable ring $R$:
	\begin{enumerate}[(1)]
		\item there exists an uncountable Polish projective $R$-module;
		\item there is an irreducible idempotent $e \in R$ such that the projective module $Re$ is of finite length as a right $eRe$-module, with the action given by multiplication;
		\item there is an irreducible idempotent $e \in R$ such that the corner ring $eRe$ is right artinian and the right $eRe$-module $Re$ is of finite length; and
		\item there is a nonzero idempotent $e \in R$ such that the corner ring $eRe$ is a left $F$-ring and $Re$ is a finitely presented right $eRe$-module.
	\end{enumerate} 
\end{theorem}

Finally, the Main Theorem allows us to characterize the existence of uncountable Polish modules with a Kaplansky decomposition. This condition on $R$ is left-right symmetric.

\begin{theorem} \label{Polish existence}
	Let $R$ be a countable ring. There exists an uncountable Polish left $R$-module with a Kaplansky decomposition if and only if there exists a nonzero indecomposable endofinite left $R$-module. Equivalently, if and only if there exists an indecomposable endofinite right $R$-module.
\end{theorem}

A class of examples where the condition of Theorem~\ref{Polish existence} fails is provided by the countable indiscrete rings of Prest, Rothmaler, and Ziegler~\cite{PRZ} which are not simple artinian~\cite[\S 8.2.12]{prest_second_book}. By definition, their Ziegler spectra are indiscrete topological spaces. Simple artinian rings are also indiscrete, but their spectra consist of a single point, represented by the unique simple module; these rings must therefore be excluded here. There are von Neumann regular examples, such as the direct limit
 $R=\varinjlim_n M_{2^n}(\grD)$ over a countable division ring $\grD$, with unital embeddings $a\mapsto\operatorname{diag}(a,a)$. There are also indiscrete examples which are not von Neumann regular; indiscrete rings belong to the broader class of \emph{almost regular} rings discussed in the same section. The indiscrete rings which are not simple artinian have more than one point in their spectra and hence no closed point. By Fact~\ref{fact_closed_point}, a countable such ring admits no indecomposable endofinite module. For the displayed direct limit, one can also exclude any nonzero endofinite module directly. Indeed, if $0\neq V$ had finite length over $S=\End_R V$, the diagonal matrix units of $M_{2^n}(\Delta)$ would decompose $V$ into $2^n$ nonzero right $S$-submodules. These submodules are isomorphic via the matrix units, whose actions commute with $S$. Thus the length of $V_S$ would be at least $2^n$ for every $n$, a contradiction. By Theorem~\ref{Polish existence}, such a countable indiscrete ring admits no uncountable Polish module with a Kaplansky decomposition.

There is also a version of Theorem~\ref{Polish existence} for a prescribed complete first-order theory $T$ of $R$-modules. Recall that the \textbf{unlimited part of $T$} is the complete theory defined as
\begin{equation}\label{def_unlimited_theory} T_u \vdash [\varphi \; \colon \psi] = 
	\left\{ \begin{array}{cc} \infty & \mbox{if} \;\; T \vdash [\grf \; \colon \psi] = \infty ;\\
		1 & \mbox{if} \;\; T \vdash [\grf \; \colon \psi] < \infty. \end{array} \right.
\end{equation} 
The theory $T_u$ is consistent, for if $M \prec N$ is an $|M|^+$-saturated extension of $M \models T,$ then $N/M \models T_u.$ The important property that we will need is that any point $V \in \Cl (T_u) \subseteq \Cl (T) \subseteq \Zg (R)$ has the property that $\grf (V) = \psi (V)$ for every pp-pair $\grf/\psi$ whose index is finite in $T.$ It follows that if $B \models T,$ then $B \oplus V \models T.$

\begin{corollary}\label{theory_polish_existence}
Let $T$ be a complete first-order theory of modules over a countable ring $R$. Then $T$ has an uncountable Polish model admitting a Kaplansky decomposition if and only if there exists a nonzero indecomposable endofinite $R$-module $V \in \Cl (T_u) \subseteq \Zg (R).$ In that case, $B\oplus V^\omega \models T$ is such a Polish model for every countable $B\models T$.
\end{corollary}

\noindent {\em Conventions.} Throughout the paper, $R$ will denote a countable ring with identity $1\neq0$, and all modules will be unital. We give $R$ the discrete topology whenever topology is considered. Unless stated otherwise, $R$-module will refer to a left $R$-module. We write $M^I$ for the direct product and $M^{(I)}$ for the direct sum of $I$ copies of $M$. We will refer to totally transcendental modules\footnote{As the ring $R$ is assumed to be countable throughout, these are just the $\gro$-stable modules.} as $\Sigma$-algebraically compact modules, but their theories as totally transcendental.

\noindent {\em Acknowledgement.} Language models (ChatGPT, OpenAI, GPT-6; Claude, Anthropic) were used to check bibliographic references and to edit the wording of parts of the exposition for clarity. All mathematical content of the paper, including every definition, statement and proof, is due to the authors, and no mathematical argument was produced by these tools. The authors reviewed and revised all edited text and take full responsibility for the final version.

 \endgroup

\section{Preliminaries: the model theory of modules}

	The language ${\mathcal L} (\RMod) = (+,-,0,r)_{r \in R}$ for $R$-modules expands the language for abelian groups with a unary function symbol for every $r \in R$; it is therefore also countable. The class $\RMod$ is elementary, axiomatized by the usual axioms ${_R}T \subseteq \Th (\RMod)$ for an $R$-module, which are first-order expressible in ${\mathcal L} (\RMod)$. A tuple $\bfu$ may be considered as a row vector of length $|\bfu|.$

\subsection{Positive-primitive formulae}

Terms in ${\mathcal L} (\RMod)$ are $0$ and linear combinations $r_1 u_1 + r_2 u_2 + \cdots + r_n u_n$ of variables with scalars in the ring, so that atomic formulae are just linear equations. 

	\begin{definition} \label{def_pp_formula} A \textbf{positive-primitive formula (pp-formula)} $\grf (\bfu)$ is an existentially quantified finite system of linear equations, displayed as
		\begin{equation} \label{pp_form} \grf (\bfu) = \exists \bfv \; (A \bfu^t \doteq B \bfv^t),\end{equation}
	with free variables $\bfu = (u_1, \ldots, u_n)$, existentially bound variables $\bfv = (v_1, \ldots, v_k)$, and matrices $A$ and $B$, with entries in $R$, of sizes $m \times n$ and $m \times k$, respectively; the superscript $t$ on a tuple of variables indicates that it is a column vector. 
	\end{definition}

Actually, we call a formula $\mrm{pp}$ if it is logically equivalent, relative to the theory ${_R}T$ of $R$-modules introduced above, to a pp-formula. For example, if $\grf_1 (\bfu)$ and $\grf_2 (\bfu)$ are pp-formulae in the same free variables, then the conjunction $(\grf_1 \wedge \grf_2) (\bfu) := \grf_1 (\bfu) \wedge \grf_2 (\bfu)$ has a positive-primitive prenex normal form. 

A pp-formula $\grf (\bfu)$ defines in the $R$-module ${_R}M$ a \textbf{pp-definable subgroup} $\grf(M) \subseteq M^{|\bfu|}.$ The assignment $M \mapsto \grf (M)$ is functorial: if $f \colon {_R}M \to {_R}N$ is a morphism of $R$-modules, then $f^{|\bfu|} \colon M^{|\bfu|} \to N^{|\bfu|}$ denotes the map obtained by applying $f$ coordinatewise, and $\grf (f) \colon \grf (M) \to \grf (N)$ is its restriction, as in the commutative diagram
\begin{equation}\xymatrix{M^{|\bfu|} \ar[r]^{f^{|\bfu|}} & N^{|\bfu|} \\
		\grf (M) \ar[r]^{\grf (f)} \ar@{^{(}->}[u] & \grf (N). \ar@{^{(}->}[u]
}\end{equation}
This functor is denoted by $\grf (-) : \RMod \to \Ab$, its $M$-component by $\grf (M)$.

Modulo equivalence over ${_R}T$, the conjunction operation just defined induces the meet operation $\wedge$ on pp-formulae. There is also a supremum operation on pp-formulae. Given $\grf (\bfu)$ and $\psi (\bfu),$ the formula 
\begin{equation}(\grf + \psi) (\bfu) := \exists \bfv, \bfw \; (\grf (\bfv) \wedge \psi (\bfw) \wedge \bfu \doteq \bfv + \bfw)\end{equation}
is also a pp-formula and it defines in the $R$-module $M$ the pp-definable subgroup \linebreak
$(\grf + \psi) (M) = \grf (M) + \psi (M)$ of $M^{|\bfu|}$.

	\begin{definition} \label{pp-pair}
		A \textbf{pp-pair} $\grf/\psi = (\grf/\psi)(\bfu)$ refers to an implication $${_R}T \vdash \psi (\bfu) \to \grf (\bfu)$$ of pp-formulae. Even if two pp-formulae $\psi$ and $\grf$ are not related by implication, one writes $\grf/\psi$ to mean $\grf/(\grf \wedge \psi)$. If $\grf/\psi$ is a pp-pair, then for every module $M$, $\psi (M) \leq \grf (M)$ and the quotient group is denoted by $(\grf/\psi)(M) := \grf (M)/\psi (M)$. This assignment $(\grf/\psi)(-) \colon M \mapsto (\grf/\psi)(M)$ is also functorial. 
	\end{definition}

\subsection{Purity}\label{sec_purity}

A short exact sequence 
\begin{equation} \label{pex}
	\xymatrix@1{0 \ar[r] & M \ar[r]^f & N \ar[r] & N/M \ar[r] & 0}
\end{equation} is \textbf{pure exact} if for every pp-formula $\grf (\bfu)$, the sequence of abelian groups
\begin{equation}\xymatrix@1{0 \ar[r] & \grf(M) \ar[r]^{\grf(f)} & \grf(N) \ar[r] & \grf(N/M) \ar[r] & 0}\end{equation}
is exact. The embedding of $R$-modules $f\colon M \hookrightarrow N$ that occurs in the pure exact sequence~(\ref{pex}) is called \textbf{pure} or a \textbf{pure embedding.} There are many equivalent conditions for a short exact sequence to be pure exact. The one that will be useful to us is the seemingly stronger one that for every pp-pair $\grf/\psi$, the sequence of abelian groups 
\begin{equation}\label{pure_pp_pair_exact}\xymatrix@1@C=30pt{0 \; \ar[r] & \; (\grf/\psi)(M) \; \ar[r]^{(\grf/\psi)(f)} & \; (\grf/\psi)(N)\; \ar[r] & \;(\grf/\psi)(N/M)\; \ar[r] &\; 0}\end{equation}
is exact. 

A module $M$ is \textbf{pure-injective} if every pure exact sequence~(\ref{pex}) is split exact, meaning that the embedding $f \colon M \hookrightarrow N$ admits a retraction $r \colon N \to M$, $rf = 1_M$ (equivalently, $M$ is a direct summand of $N$).

Pure-injective modules are equivalently \textbf{algebraically compact}: every system of linear equations with constants in the module whose finite subsystems are solvable has a simultaneous solution~\cite[Theorem~3.1]{Zg}.\bigskip 

For $a\in M$, write $\tp^+(M,a):=\{\grf(u):\grf\text{ is pp and }M\models\grf(a)\}$ for its \textbf{positive pp-type} over the empty set. If there is a morphism $f \colon (M,a) \to (N,b)$ of pointed $R$-modules, i.e., a morphism $f \colon {_R}M \to {_R}N$ of $R$-modules such that $f(a) = b,$ then $\tp^+ (M,a) \subseteq \tp^+ (N,b).$ Conversely, if $N$ is pure-injective, and $\tp^+ (M,a) \subseteq \tp^+ (N,b),$ then there exists a morphism $f: (M,a) \to (N,b)$~\cite[Corollary~3.3]{Zg}. The pure-injective hull $H(a) = H(\tp^+(N,a))$ is a pure-injective direct summand of $N$ containing $a$ with the property that the restriction to $H(a)$ of any morphism $g: (N,a) \to (K,c)$ with $\tp^+ (N,a) = \tp^+ (K,c)$ is a pure embedding, and therefore split. These considerations lead to the following fact.

\begin{fact}[{\cite[Theorem~3.6 and Lemma~4.1]{Zg}}]\label{pointed_hull_fact}
Let $N$ and $N'$ be indecomposable pure-injective modules, and let $0\neq a\in N$ and $0\neq a'\in N'$. If $\tp^+(N,a)=\tp^+(N',a')$, then there is an isomorphism $N\cong N'$ taking $a$ to $a'$.
\end{fact}

\begin{proof}
The pointed pure-injective hull of $a$ is a direct summand of $N$. Since $a\neq0$ and $N$ is indecomposable, this hull is all of $N$, by~\cite[Lemma~4.1]{Zg}. Likewise, $N'$ is the pointed pure-injective hull of $a'$. The equal positive pp-types identify the two pointed hulls, by the uniqueness assertion in~\cite[Theorem~3.6]{Zg}, giving an isomorphism taking $a$ to $a'$.
\end{proof}

\subsection{Complete theories of modules}

If $\grf/\psi$ is a $\mrm{pp}$-pair, then the sentence 
\begin{equation}[\grf \; \colon \psi] \geq n := \exists \bfu_1, \ldots, \bfu_n \; (\bigwedge_i \grf (\bfu_i) \wedge \bigwedge_{i<j} \neg \psi (\bfu_i - \bfu_j))\end{equation}
is first-order expressible in the language ${\mathcal L}(\RMod)$. It follows that if ${_R}M \equiv {_R}N$ are elementarily equivalent $R$-modules, then the groups 
$(\grf/\psi)(M)$ and $(\grf/\psi)(N)$ are either both of the same finite order or are both infinite. This is denoted by  
\begin{equation}\label{BGM} [\grf (M) \; \colon  \psi (M)] \; \equiv \; [\grf (N) \; \colon  \psi (N)] \!\!\!\! \pmod{\infty}.\end{equation}
The Baur--Monk theorem (see the original proof in~\cite{Baur} and~\cite[Corollary~2.15]{prest_first_book}) gives the converse: if $M$ and $N$ are $R$-modules such that condition~\eqref{BGM} holds for every pp-pair $\grf/\psi$, then $M \equiv N$. \bigskip

If $M \subseteq N$ is a pure submodule, then the complete theory of the quotient module $N/M$ can sometimes be determined using the following proposition.

\begin{proposition}[{\cite[Proof of Corollary~2.2(1)]{Zg}}]\label{Ziegler Form 2} If the short exact sequence~(\ref{pex})
	is pure exact, then for every pp-pair $\grf/\psi$, 
	\begin{equation} [\grf (N) \; \colon  \psi (N)] = [\grf (M) \; \colon  \psi (M)] \cdot [\grf (N/M) \; \colon  \psi (N/M)].\end{equation}
\end{proposition}

\begin{definition}[{\cite{garavaglia}}]\label{gava_ex}
	A module $M$ is said to be $\Sigma$\textbf{-algebraically compact} if the direct sum of any number of copies of $M$ is pure-injective. $\Sigma$-algebraically compact modules are characterized by the descending chain condition on pp-definable subgroups. In other words, given a descending chain of pp-definable subgroups,
	\begin{equation}M^n = \grf_0 (M) \supseteq \grf_1 (M) \supseteq \cdots \supseteq \grf_i (M) \supseteq \cdots,\end{equation}
	there is an $m$, such that for all $k \geq m$, $\grf_k (M) = \grf_{k+1}(M)$. Replacing each $\grf_k$ by $\bigwedge_{i \leq k}\grf_i$ does not change its $M$-component, so we may take each $\grf_k/\grf_{k+1}$ to be a pp-pair and say that $\Th (M) \models [\grf_k \; \colon \grf_{k+1}] = 1$, for all $k \geq m$.
\end{definition}

\subsection{The Ziegler spectrum}\label{sec_ziegler}

The Ziegler spectrum $\Zg(R)$ of a ring $R$ is the topological space whose points are the (isomorphism types of) indecomposable pure-injective $R$-modules. A basis for its topology is given by the sets
\[
	\Op(\grf/\psi):=\{U\in\Zg(R)\mid(\grf/\psi)(U)\neq 0\},
\]
as $\grf/\psi$ ranges over pp-pairs.
Given a complete theory $T\supseteq {_R}T$, let $\Cl(T)\subseteq\Zg(R)$ be defined by specifying its complement:
\[
	\Cl(T)^c:=\bigcup \; \{\Op(\grf/\psi)\mid T\models[\grf \; \colon \psi]=1\}.
\]
Since every $\Op(\grf/\psi)$ is open, $\Cl(T)^c$ is open. Thus $\Cl(T)$ is a closed subset of $\Zg(R)$, called {\em the closed subset associated with $T$}. For a module $M$, we define $\Cl(M):=\Cl(\Th(M))$. For $M\models T$, put $T^{\gro}:=\Th(M^{\gro})$. Since $(\grf/\psi)(M^{\gro})\cong((\grf/\psi)(M))^{\gro}$, its pp-invariants are $1$ where those of $T$ are $1$, and infinite otherwise; Baur--Monk therefore shows that this definition is independent of $M$. Evidently, $\Cl(T)=\Cl(T^{\gro})$. Finally, associated with a closed set $\cC\subseteq\Zg(R)$ is the module $M_{\cC}:= \; \bigoplus \;\{U\mid U\in\cC\}$.

\begin{fact}[{\cite[Theorem~4.9 and Corollary~4.10]{Zg},~\cite[\S 5.1]{prest_second_book}}]\label{Ziegler_topology}
	The rule $T \mapsto \Cl (T)$ is a bijective correspondence between complete theories of modules for which every pp-invariant is either $1$ or infinite and closed subsets $\cC \subseteq \Zg (R)$; the inverse rule is given by $\cC \mapsto \Th((M_{\cC})^{\gro})$.
\end{fact}

\begin{definition}
	A closed subset $\cC \subseteq \Zg (R)$ is \textbf{totally transcendental} if $M_{\cC}$ is $\Sigma$-algebraically compact. Equivalently, any theory $T$ for which $\cC = \Cl (T)$ satisfies the descending chain condition on pp-formulae (since $T^{\gro}=\Th((M_{\cC})^{\gro})$, and a descending chain of pp-definable subgroups stabilizes in $M_{\cC}$ if and only if it stabilizes in any nonempty power of $M_{\cC}$).
\end{definition}

\begin{lemma} \label{tt_closed_set}
	If $R$ is countable, then every totally transcendental closed subset $\cC \subseteq \Zg (R)$ is countable and every point $U \in \cC$ is a countable indecomposable module. The module $M_{\cC}$ is therefore countable.
\end{lemma}

\begin{proof}
	To see that $\cC$ is countable, let $T \supseteq {_R}T$ be a complete theory for which $\cC = \Cl (T)$. Then $T$ is totally transcendental and if $U \in \cC$ and $a \in U$ is nonzero, there is a pp-formula $\grf (u)$ such that 
	$$T \vdash \tp^+ (U,a) \leftrightarrow \grf (u).$$ Since there are only countably many pp-formulae, there can only be countably many pp-types $p(u)$, relative to $T$, and every $U \in \Cl (T)$ is the pure-injective hull $H(p)$ of such a pp-type.  
	
\smallskip \noindent 	Every $U \in \cC$ is $\Sigma$-algebraically compact, so if $V \prec U$ is a countable elementary submodule, then it too is $\Sigma$-algebraically compact. But an elementary embedding is pure, so that $V$ is necessarily a direct summand of the indecomposable $U$, whence $U = V$ is countable. 
\end{proof}

Let ${_R}M$ be a module and put $S=\End_R M$. We write multiplication in $S$ in the order $fg=g\circ f$, so that $af=f(a)$ defines a right action on $M$. Endomorphisms preserve pp-formulae, and hence every pp-definable subgroup $\grf(M)\subseteq M^n$ is a right $S$-submodule under the coordinatewise action. The same is true of the quotient $(\grf/\psi)(M)$ associated with a pp-pair.

A module $V$ is \emph{endofinite} if its \emph{endolength}, the length of $V$ as a right module over $S=\End_R V$, is finite. Finite length means that there is a finite composition series $0=V_0\subsetneq V_1\subsetneq\cdots\subsetneq V_\ell=V$ of right $S$-submodules, with each quotient $V_i/V_{i-1}$ simple (nonzero and with no proper nonzero submodule). Equivalently, $V_S$ satisfies both the ACC and the DCC on submodules. In particular, every strict chain of submodules has finitely many inclusions.

We will see, in Section~\ref{sec_main_theorem}, that the existence of an uncountable Polish module with a Kaplansky decomposition is equivalent to the existence of a nonempty totally transcendental closed subset $\cC \subseteq \Zg (R)$. We collect some well-known equivalent conditions needed below. The first equivalence of the next fact is the $\mrm{CB}$-$\mrm{rk}=0$ case of a deep result~\cite[8.6]{Zg} of Ziegler that equates the $\mrm{CB}$-rank~\cite[\S 6.C]{K1} of $\Op (\grf/\psi) \cap \Cl (T)$ and the $m$-dimension of the interval $[\psi, \grf]_T$ in the lattice of pp-formulae, modulo logical equivalence relative to $T$; it is not known if the forward direction holds when the ring $R$ is uncountable. 

\begin{fact}[{\cite[8.6]{Zg},~\cite[\S 4.4.3]{prest_second_book}}]\label{fact_closed_point} Let $R$ be countable. The following are equivalent for a point $U \in \Zg (R)$:
	\begin{enumerate}[(1)]
		\item $U$ is closed, i.e., the singleton $\{ U \} \subseteq \Zg (R)$ is a closed set;
		\item the module ${_R}U$ has the ACC and the DCC on pp-definable subgroups;
		\item ${_R}U$ is endofinite, i.e., of finite length as a right module over the endomorphism ring $S = \End_R U$.
	\end{enumerate} 
\end{fact}

The implication ($1$) $\Rightarrow$ ($2$) of Fact~\ref{fact_closed_point} implies that if $U \in \Zg (R)$ is a closed point, then the closed subset $\{ U \} \subseteq \Zg (R)$ is totally transcendental. The converse is another application of Ziegler's Theorem~\cite[8.6]{Zg}, which implies more generally that if $T$ is a complete theory in which there does not exist a dense chain of pp-definable subgroups, then $\Cl (T)$ contains a closed point. A totally transcendental theory clearly satisfies this condition so we can state this as follows. \medskip

\begin{fact}[{\cite[8.6]{Zg}}]
	There exists a nonempty totally transcendental closed set $\cC \subseteq \Zg (R)$ if and only if there exists a closed point $U \in \Zg (R)$.
\end{fact}

Finally, let us note that the existence of an endofinite point $U \in \Zg (R)$ is a left-right symmetric condition for the ring $R$.
Elementary duality~\cite[Theorem~6.6]{Herzog_duality} associates to a complete theory $T$ of left $R$-modules a complete theory $T'$ of right $R$-modules. If $D$ denotes the Prest duality~\cite[\S 1.3.1]{prest_second_book} on pp-formulae, then
\[
T\vdash\psi\to\grf
\quad\Longleftrightarrow\quad
T'\vdash D\grf \to D\psi,
\]
and applying the duality twice returns $T$.
\smallskip

\begin{proposition}\label{endofinite_left_right}
The Ziegler spectrum $\Zg(R)$ contains an endofinite point if and only if $\Zg(R^{\op})$ does. Indeed, the closed set $\Cl (T) \subseteq \Zg (R)$ of a complete theory $T$ of left $R$-modules contains an endofinite point if and only if the closed subset $\Cl (T') \subseteq \Zg (R^{\op})$ does. 
\end{proposition}

\begin{proof}
 Let $U \in \Cl (T)$  be an endofinite point and denote by $S = \Th (U)$ its complete theory. If $T \models [\grf \, \colon \psi] = 1,$ then obviously $\grf (U) = \psi (U),$ so that we also have that $S \models [\grf \, \colon \psi] = 1,$ whence $\Cl (U) = \Cl (S) \subseteq \Cl (T).$
 The pp-lattices modulo $S$ and $S'$ are order anti-isomorphic, so the ACC and DCC are interchanged; by Fact~\ref{fact_closed_point}, both chain conditions hold modulo $S$, and hence also modulo $S'$. The least and greatest pp-formulae in one variable are distinct modulo $S$; their duals are therefore distinct modulo $S'$, so a model $N\models S'$ is nonzero. Its DCC makes $N$ $\Sigma$-algebraically compact. Garavaglia's decomposition theorem~\cite{garavaglia} gives a nonzero indecomposable pure-injective summand $U'$ of $N$. Both chain conditions pass to $U'$, and Fact~\ref{fact_closed_point}, applied to $R^{\op}$, makes $U'$ an endofinite point of $\Cl (S') \subseteq \Cl (T').$ The same argument with $R^{\op}$ proves the converse.

\end{proof}

\subsection{The endomorphism action and endolength}\label{sec_endomorphisms}

We collect further facts about the endomorphism action and finite endolength used in the existence results of Section~\ref{sec_applications}, especially Subsections~\ref{sec_projective} and~\ref{sec_kaplansky_existence}.
Let ${_R}M$ be a module and put $S=\End_R M$, with the multiplication convention fixed in Subsection~\ref{sec_ziegler}.

\begin{lemma}\label{endo_fg_projection}
Suppose that, for every quantifier-free pp-formula $\gra(\bfv)$, the right $S$-module $\gra(M)$ is finitely generated. Then $\grf(M)$ is finitely generated as a right $S$-module for every pp-formula $\grf(\bfu)$.
\end{lemma}

\begin{proof}
Write $\grf(\bfu)=\exists\bfv\,\gra(\bfu,\bfv)$ with $\gra$ quantifier-free. The projection $\pi\colon\gra(M)\to\grf(M)$, $(\bfa,\bfb)\mapsto\bfa$, is surjective. For $f\in S$ we have $\pi((\bfa,\bfb)f)=\bfa f=\pi(\bfa,\bfb)f$, so $\pi$ is $S$-linear. Thus $\grf(M)$ is a quotient of the finitely generated right $S$-module $\gra(M)$, and is itself finitely generated.
\end{proof}

For an endofinite module $V$, put $S=\End_R V$. If $V_S$ has length $\ell$, then $(V^n)_S$ has length $n\ell$. Its pp-definable subgroups are $S$-submodules, so every descending chain of them stabilizes. Hence an endofinite module is $\Sigma$-algebraically compact. In particular, an indecomposable endofinite module is pure-injective and has a local endomorphism ring~\cite[Theorem~4.3]{Zg}.

\begin{lemma}\label{tt_endofinite_point}
Let $R$ be countable. Every nonempty totally transcendental closed subset $\cC\subseteq\Zg(R)$ contains an indecomposable endofinite module.

\end{lemma}

\begin{proof}
The Ziegler spectrum is quasi-compact~\cite[Theorem~4.9]{Zg}, so its closed subspace $\cC$ is quasi-compact. Every descending chain of nonempty closed subsets of $\cC$ has nonempty intersection. Zorn's lemma therefore gives a minimal nonempty closed subset $\mathcal D\subseteq\cC$.

\smallskip\noindent
Every closed subset of a totally transcendental closed set is again totally transcendental: its module of representatives is a direct summand of $M_{\cC}$, and thus inherits the DCC on pp-definable subgroups. Ziegler's rank theorem~\cite[Lemma~8.4(2) and Theorem~8.6]{Zg} implies that the nonempty space $\mathcal D$ has an isolated point $V$. If $\mathcal D\setminus\{V\}$ were nonempty, it would be a proper nonempty closed subset of $\mathcal D$, contrary to minimality. Hence $\mathcal D=\{V\}$ is closed in $\Zg(R)$. Fact~\ref{fact_closed_point} now shows that $V$ is endofinite.
\end{proof}

\section{The model theory of Polish modules}

\begin{definition}\label{def:polish-module}
	Let $R$ be a countable ring, endowed with the discrete topology. A
	\emph{Polish $R$-module} is an $R$-module $A$ equipped with a topology
	$\tau$ such that:
	\begin{enumerate}[(1)]
		\item $(A,\tau)$ is a Polish space (i.e., it is separable and completely metrizable);
		\item $(A,+,-,0,\tau)$ is a topological group;
		\item the map $R\times A\to A\colon (r,x)\mapsto rx$ is continuous.
	\end{enumerate}
\end{definition}

The solution set of a quantifier-free pp-formula is the kernel of a continuous homomorphism between finite powers of a Polish module, in more detail:

\begin{lemma} \label{pp_qf} If $A$ is a Polish $R$-module and $\gra (\bfu)$ a quantifier-free pp-formula, then $\gra(A)$ is a closed subgroup of $A^{|\bfu|}$.
\end{lemma}

\begin{proof}
	Write $|\bfu|=n$. Since $\gra(\bfu)$ is a quantifier-free pp-formula,
	it is a finite conjunction of homogeneous linear equations. Thus there are
	$m<\omega$ and a matrix $C\in M_{m\times n}(R)$ such that
	\begin{equation}
		\gra(A)=\ker(C_A),
	\end{equation}
	where $C_A\colon A^n\to A^m$ is given by
	$C_A(\bfa)=C\bfa^t$. The map $C_A$ is continuous because addition and
	scalar multiplication are continuous on $A$. Since $A^m$ is Hausdorff,
	$\{\bfzero\}$ is closed, and hence
	$\gra(A)=C_A^{-1}(\{\bfzero\})$ is closed in $A^n$. It is a subgroup because the solution set of a homogeneous system contains zero and is closed under addition and additive inverses.
\end{proof}

Because every pp-formula is an existentially quantified system of linear equations, we obtain the following. \medskip

\begin{lemma}\label{pp_analytic} If $A$ is a Polish $R$-module and $\varphi(\bfu)$ a pp-formula, then $\varphi(A)$ is an analytic subgroup of $A^{|\bfu|}$. 

\end{lemma}

\begin{proof}
	Write $\varphi(\bfu)$ in the form
	\begin{equation}
		\exists \bfv\,\alpha(\bfu,\bfv),
	\end{equation}
	where $\alpha$ is a quantifier-free pp-formula. By
	Lemma~\ref{pp_qf}, $\alpha(A)$ is a closed subgroup of the relevant finite
	power of $A$, and is therefore a Polish space. The coordinate projection
	$(\bfu,\bfv)\mapsto\bfu$ is continuous and has image $\varphi(A)$.
	Consequently, $\varphi(A)$ is analytic, by the characterization of analytic sets as continuous images of Polish spaces, cf. e.g.~\cite[Exercise~14.3]{K1}. It is a subgroup because pp-formulae
	define subgroups in every module.
\end{proof}

\subsection{Polish topologies on pp-definable subgroups}

The next lemma lets us apply Polish-group arguments directly to pp-definable subgroups. It also strengthens Lemma~\ref{pp_analytic} by showing that these subgroups are Borel. We first recall the quotient theorem used in its proof.

\begin{fact}[{\cite[Theorem~1.17]{Melleray}}]\label{polish_group_quotient}
Let $G$ be a Polish group and let $K$ be a closed normal subgroup. Write $q_K\colon G\to G/K$ for the quotient homomorphism. Then $G/K$, endowed with the quotient topology
\[
 \tau_{G/K}:=\{W\subseteq G/K:q_K^{-1}(W)\text{ is open in }G\},
\]
is a Polish group.
\end{fact}

If $h\colon G/K\to H$ is an isomorphism of abstract groups, \emph{transporting} $\tau_{G/K}$ to $H$ means declaring $V\subseteq H$ open exactly when $h^{-1}(V)$ is open in $\tau_{G/K}$. This equips $H$ with a Polish group topology for which $h$ is a homeomorphism. Here and below we work with additive abelian groups: a pp-definable subgroup need not be an $R$-submodule when $R$ is noncommutative. In particular, the topology placed on $H=\grf(A)$ in the next lemma is a Polish topology on its additive group. For a topology $\tau$ on $A$, we write $\tau^n$ for the product topology on $A^n$.

\begin{lemma}\label{pp_polish}
Let $(A,\tau)$ be a Polish $R$-module and let $\grf(\bfu)=\exists\bfv\,\gra(\bfu,\bfv)$ be a pp-formula, with $\gra$ quantifier-free and $|\bfu|=n$, $|\bfv|=k$. Put
\[
 G=\gra(A),\qquad
 K=\{(\bfzero,\bfb)\in G:\bfb\in A^k\},\qquad H=\grf(A),
\]
and give $G$ the subspace topology from $A^{n+k}$. The projection $\pi\colon G\to H$, $(\bfa,\bfb)\mapsto\bfa$, induces an isomorphism of abelian groups
\[
 \overline\pi\colon G/K\longrightarrow H,\qquad
 (\bfa,\bfb)+K\longmapsto\bfa.
\]
Transporting $\tau_{G/K}$ along $\overline\pi$ gives a Polish group topology $\tau_q$ on $H$. Equivalently, $V\subseteq H$ is $\tau_q$-open if and only if $\pi^{-1}(V)$ is open in $G$. Moreover:
\begin{enumerate}[(1)]
\item $\tau_q$ is finer than or equal to the subspace topology $\tau_{\mathrm{sub}}$ induced on $H$ by $(A^n,\tau^n)$;
\item for every $D\subseteq H$, the set $D$ is analytic in $(H,\tau_q)$ if and only if it is analytic as a subset of $(A^n,\tau^n)$;
\item $H$ is a Borel subset of $(A^n,\tau^n)$;
\item the inclusion $i\colon(H,\tau_q)\to(A^n,\tau^n)$, $i(\bfa)=\bfa$, is a continuous injection between Polish spaces.
\end{enumerate}

\end{lemma}

\begin{proof}
By Lemma~\ref{pp_qf}, $G$ is a closed subgroup of $A^{n+k}$ and is therefore Polish. Since $A^n$ is Hausdorff, $\{\bfzero\}$ is closed in $A^n$; thus $\{\bfzero\}\times A^k$ is closed in $A^{n+k}$, and $K=G\cap(\{\bfzero\}\times A^k)$ is closed in $G$. The group $G$ is abelian, so $K$ is normal. The map $\pi$ is a surjective homomorphism of abelian groups with kernel $K$. The first isomorphism theorem therefore gives the stated isomorphism $\overline\pi$, with $\pi=\overline\pi\circ q_K$.

\smallskip\noindent
By Fact~\ref{polish_group_quotient}, $(G/K,\tau_{G/K})$ is Polish. Its transported topology $\tau_q$ makes $\overline\pi$ a homeomorphism. For $V\subseteq H$, the definitions and $\pi=\overline\pi\circ q_K$ give
\[
 \begin{split}
 V\in\tau_q
 &\quad\Longleftrightarrow\quad \overline\pi^{-1}(V)\in\tau_{G/K}\\
 &\quad\Longleftrightarrow\quad q_K^{-1}\bigl(\overline\pi^{-1}(V)\bigr)=\pi^{-1}(V)
       \text{ is open in }G.
 \end{split}
\]

\smallskip\noindent
If $O\subseteq A^n$ is open, then $O\times A^k$ is open in $A^{n+k}$, so
$\pi^{-1}(H\cap O)=G\cap(O\times A^k)$ is open in the subspace $G$. The preceding equivalence shows that $H\cap O$ is $\tau_q$-open. Every $\tau_{\mathrm{sub}}$-open set has this form, proving $\tau_{\mathrm{sub}}\subseteq\tau_q$ in (1). Also $i^{-1}(O)=H\cap O$, so $i$ is continuous. It is injective, and its domain and codomain are Polish, proving~(4).

\smallskip\noindent
If $D$ is analytic in $(H,\tau_q)$, its continuous image $i(D)=D$ is analytic in $A^n$, by~\cite[Proposition~14.4]{K1}. Conversely, if $D\subseteq H$ is analytic in $A^n$, its continuous inverse image $i^{-1}(D)=D$ is analytic in the Polish space $(H,\tau_q)$, again by~\cite[Proposition~14.4]{K1}. This proves (2). Finally, the continuous injection in (4) has Borel image by the Lusin--Souslin theorem~\cite[Theorem~15.1]{K1}, proving (3).
\end{proof}

\begin{remark}
The topology $\tau_q$ on $H=\grf(A)$ need not equal the inherited topology, by which we mean the subspace topology $\tau_{\mathrm{sub}}$ induced from $A^n$. In the proofs below, all assertions about openness, meagerness, or the Baire property inside $\grf(A)$ refer to $\tau_q$. Lemma~\ref{pp_polish}(2) allows us to use subgroups originally known to be analytic in $A^n$. We are equipping an auxiliary abelian group with a topology; no change is made to the topology of the module $A$. We do not claim that $\tau_q$ is an $R$-module topology: $H$ need not even be an $R$-submodule.

\end{remark}

\subsection{The Polish unlimited theory and the bin}

The following definition is inspired by what is known as the unlimited part $T_u,$ defined in the introduction, of a complete theory $T$ of modules; see~\cite[\S~4.5]{prest_first_book}. It adapts that construction by replacing the finite-index cutoff with a countable-index cutoff.

Given a module $A \models {_R}T$, let us introduce a general definition of $\Th_{\Pun}(A)$, the \textbf{Polish unlimited theory} of $A$, which will have the property stated in the introduction when $A$ is a Polish $R$-module. 

\begin{equation}\label{def_polish_unlimited_theory}\Th_{\Pun} (A) \vdash [\varphi \; \colon \psi] = 
	\left\{ \begin{array}{ll} \infty & \mbox{if} \;\; [\grf (A) \; \colon \psi (A)] > \aleph_0 ;\\
		1 & \mbox{if} \;\; [\grf (A) \; \colon \psi (A)] \leq \aleph_0. \end{array} \right.
\end{equation}
Observe that if $A$ is countable, then $\Th_{\Pun} (A) \models [u \doteq u \; \colon u \doteq 0] = 1$, so that $\Th_{\Pun} (A) = \Th (0)$, the complete theory of the $0$ module.  To prove that $\Th_{\Pun}(A)$ is consistent, regardless of whether it is Polish or not, we need the important notion of a {\em bin.}

\begin{definition} \label{the bin}
	Let $R$ be countable and let $A$ be any $R$-module. A \textbf{bin} $B \preccurlyeq A$ is a countable elementary submodule with the property that for every pp-pair $\grf/\psi$ for which $|\grf (A)/\psi (A)| \leq \aleph_0$, the embedding $i \colon B \preccurlyeq A$ induces an isomorphism 
	\begin{equation}(\grf/\psi)(i) \colon (\grf/\psi)(B) \to (\grf/\psi)(A)\end{equation} of abelian groups. Recall that an elementary inclusion is pure: $\grf(B)=B^{|\bfu|}\cap\grf(A)$ for every pp-formula $\grf(\bfu)$. By~\eqref{pure_pp_pair_exact}, the induced map on each pp-pair is injective. To make it surjective when the index in $A$ is countable, choose representatives of all the relevant cosets and include their coordinates in $B$. This is possible because there are only countably many pp-pairs, and, for each such $\grf/\psi$, whose $A$-component is countable, there are only countably many cosets of $\psi (A)$ in $\grf (A)$. Now use the L\"{o}wenheim-Skolem Theorem to obtain a bin $B \preccurlyeq A$.

\end{definition}

\begin{remark}
	Informally, a bin is a countable elementary submodule containing representatives of every coset belonging to a countable pp-pair of $A$. This is exactly what makes the quotient satisfy $\Th_{\Pun}(A)$ below. The simplest example is $B=A$ when $A$ is countable. On the other hand, suppose that $A$ is a \textbf{Polish unlimited module,} i.e., that $\Th_{\Pun} (A) = \Th (A)$. Then a bin is just any countable elementary submodule.
\end{remark}

\begin{proposition}\label{bin_quotient}
If $B\preccurlyeq A$ is a bin, then $A/B\models\Th_{\Pun}(A)$. In particular, $\Th_{\Pun}(A)$ is consistent.
\end{proposition}

\begin{proof}
For every pp-pair $F=\grf/\psi$, purity gives, by~\eqref{pure_pp_pair_exact}, an exact sequence $0\to F(B)\to F(A)\to F(A/B)\to0$.
\noindent If $F(A)$ is countable, the bin condition makes $F(B)\to F(A)$ an isomorphism, so $F(A/B)=0$. If $F(A)$ is uncountable, its quotient by the countable group $F(B)$ is uncountable. These are exactly the pp-invariants prescribed by $\Th_{\Pun}(A)$.
\end{proof}

Recall that pp-formulae respect products and direct sums, so if $M$ is an $R$-module and $\grf (\bfu)$ a pp-formula, then $\grf (M^I) = (\grf (M))^I$ and 
$\grf (M^{(I)}) = (\grf (M))^{(I)}$, for any index set $I$.

\begin{proposition}\label{discrete_product}
Let $M$ be a countable $R$-module. Then $M^\omega$ and $M^{(\mathfrak{c})}$ are Polish unlimited and elementarily equivalent. Moreover, the product of the discrete topologies on $M$ makes $M^\omega$ a Polish $R$-module.
\end{proposition}

\begin{proof}
For every pp-pair $F=\grf/\psi$, the product identity $F(M^{\gro})\cong F(M)^{\gro}$ shows that $F(M^{\gro})=0$ when $F(M)=0$, and that $|F(M^{\gro})|=\mathfrak{c}$ otherwise, since $F(M)$ is countable. Thus every pp-pair has index either $1$ or $\mathfrak{c}$ in $M^{\gro}$, which is exactly the assertion $\Th(M^{\gro})=\Th_{\Pun}(M^{\gro})$. The product of the discrete topologies on $M$ also makes $M^{\gro}$ a Polish $R$-module.

\smallskip \noindent
For the direct sum, $F(M^{(\mathfrak{c})})\cong F(M)^{(\mathfrak{c})}$ likewise has cardinality $1$ when $F(M)=0$, and $\mathfrak{c}$ otherwise. Hence $\Th(M^{(\mathfrak{c})})=\Th_{\Pun}(M^{(\mathfrak{c})})$. Moreover, for every pp-pair $F$, the groups $F(M^\omega)$ and $F(M^{(\mathfrak{c})})$ are both zero exactly when $F(M)=0$, and otherwise both are infinite. Thus all corresponding pp-indices agree modulo $\infty$, so Baur--Monk gives $\Th(M^{(\mathfrak{c})})=\Th(M^{\gro})$.

\end{proof}

The next theorem is the first index estimate needed for Theorem~\ref{Thm_CH}: consistency with $\Th_{\Pun}(A)$ prevents an ascending union of pp-definable subgroups from having countable index in its pp-definable upper bound.

\begin{theorem}\label{Thm_Second_qf}
	Let $R$ be a countable ring, ${_R}A$ an uncountable Polish module and consider an ascending chain of pp-definable subgroups
	\begin{equation}\psi_0 (A) \subseteq \psi_1 (A) \subseteq \cdots \subseteq \psi_m (A) \subseteq \cdots \subseteq \gra (A) \subseteq A^n\end{equation}
	that is bounded above by a pp-definable subgroup $\gra (A)$. If the pp-type 
	\begin{equation}q(\bfu) = \{ \gra (\bfu) \} \cup \{ \neg \psi_m (\bfu) \colon m < \omega \}\end{equation}
	is $\Th_{{\rm Pun}} (A)$-consistent, then the index $[\gra (A) \; \colon \Sigma_i \, \psi_i (A)] > \aleph_0$.

\end{theorem}

\begin{proof}
Replacing $\psi_i$ by $\gra\wedge\psi_i$ leaves each $\psi_i(A)$ unchanged and does not change the type $q = q(\bfu)$, since $q$ already contains $\gra$. We may therefore assume that $\psi_i\leq\gra$ holds in every module. Give $H=\gra(A)$ the Polish group topology of Lemma~\ref{pp_polish}. Each $\psi_i(A)$ is analytic in $A^n$ by Lemma~\ref{pp_analytic}, and hence analytic in $H$ by Lemma~\ref{pp_polish}(2).

\smallskip\noindent
Put $D=\bigcup_{i<\omega}\psi_i(A)=\sum_{i<\omega}\psi_i(A)$; the union is a subgroup because the chain is increasing. Suppose, towards a contradiction, that $[H:D]\leq\aleph_0$. Choose a sequence $(c_j)_{j<\omega}$ of representatives for its cosets, repeating representatives if there are only finitely many. Then
\[
 H=\bigcup_{j<\omega}\bigcup_{i<\omega}(c_j+\psi_i(A)).
\]
Since $H$ is a nonempty Polish space, the Baire Category Theorem~\cite[8.4]{K1} implies that one of these translates is nonmeager. Translation is a homeomorphism, so $\psi_i(A)$ is nonmeager for some $i < \omega$. This subgroup $\psi_i(A)$ has the Baire property because it is analytic~\cite[Theorem~21.6]{K1}.
 Pettis' Theorem~\cite[Theorem~9.9]{K1} now implies that $\psi_i(A)-\psi_i(A)=\psi_i(A)$ contains an open neighborhood $V$ of zero. Since $\psi_i(A)=\bigcup_{a\in\psi_i(A)}(a+V)$, the subgroup is open, as in~\cite[Exercise~9.11]{K1}.
 Its cosets are disjoint nonempty open sets in a separable space, so $[\gra(A):\psi_i(A)]\leq\aleph_0$.

\smallskip\noindent
The definition of the Polish unlimited theory now gives $\Th_{\Pun}(A)\models[\gra:\psi_i]=1$. Thus $\gra(\bfu)\to\psi_i(\bfu)$ holds in every model of $\Th_{\Pun}(A)$. This contradicts the assumed consistency of $q$, which contains both $\gra$ and $\neg\psi_i$. \mbox{Thus, $[\gra(A):D]>\aleph_0$.}
\end{proof}

The next example exhibits an index of the form considered in Theorem~\ref{Thm_Second_qf} that occurs in O'Neill's algebraic criterion~\cite{ON84,oneill} for $R^{\gro}$ to be free. For a countable ring $R$, the natural product topology makes $R^{\gro}$ a Polish module, so the theorem applies to this setting. In Section~\ref{sec_applications}, we compute this index for countable rings and use it to compare freeness and projectivity of $R^\omega$.

\begin{example} \label{O'Neill}
		 Let $R$ be a ring, $f \in R$ an idempotent element and $J = J(R) \subseteq R$ the Jacobson radical. We have that
		 \begin{equation}|(fR^{\gro} + JR^{\gro})/JR^{\gro}| = |fR^{\gro}/(fR^{\gro} \cap JR^{\gro})|.\end{equation}
		 If $R$ is countable, then $A = R^{\gro}$ is a Polish module, and the pp-subgroup $fR^{\gro} = fA$ is defined by the quantifier-free formula $(1-f)u \doteq 0$. If we express the ideal $J = \Sigma_{i< \gro} \; J_i$ as an increasing union of finitely generated right ideals $J_i$, then each of the subgroups $fR^{\gro} \cap J_iR^{\gro} = fA \cap J_iA$ is defined by the pp-formula
		 \begin{equation}(1-f)u \doteq 0 \wedge \exists v_1, v_2, \ldots v_n (u \doteq r_1v_1 + \cdots + r_n v_n),\end{equation} where the $r_j$'s generate $(J_i)_R$.

	\end{example}

\subsection{The index dichotomy}

Recall that a topological space is \emph{perfect} if it has no isolated points. We will use the following form of Mycielski's theorem for perfect Polish spaces.

\begin{lemma}[{Mycielski; see \cite[Theorem~19.1]{K1} and \cite[p.~2719]{BZ13}}]\label{mycielski_lemma} Let $X$ be a non-empty perfect Polish space and let $R \subseteq X \times X$ be meager in $X \times X$. Then there is a Cantor set $C \subseteq X$ (i.e., a homeomorphic copy of the Cantor space $2^\omega$) such that, for all $x \neq y \in C$, we have that $(x, y) \notin R$.

\end{lemma}

\begin{theorem}\label{analytic_dichotomy} Let $A$ be a Polish $R$-module, let $\varphi(\bfu)$ be a pp-formula with $|\bfu|=n$ and let $N\leq\varphi(A)\leq A^n$ be an analytic subgroup of $A^n$. Then the index $[\varphi(A) : N]$ is either countable or equal to $\mathfrak{c}$.
\end{theorem}

\begin{proof}
Put $H=\varphi(A)$ and give it the Polish group topology of Lemma~\ref{pp_polish}. The subgroup $N$ is analytic in this topology by part~(2) of that lemma. Since $H\subseteq A^n$, we have $[H:N]\leq\mathfrak{c}$. If $H$ is countable, the conclusion is immediate, so assume that $H$ is uncountable. It then has no isolated points: otherwise translations would make it discrete, whereas a dense subset of a discrete space is the whole space, so a separable discrete space is countable. Thus $H$ is a nonempty perfect Polish space.

\smallskip\noindent
Since $N$ is analytic in $H$, it has the Baire property by~\cite[Theorem~21.6]{K1}. If $N$ is nonmeager in $H$, Pettis' Theorem~\cite[Theorem~9.9]{K1} implies that $N-N=N$ contains an open neighborhood $V$ of zero. Then $N=\bigcup_{a\in N}(a+V)$ is open, as in~\cite[Exercise~9.11]{K1}. Its cosets are disjoint nonempty open sets, each meeting a fixed countable dense subset of $H$, so there are only countably many cosets.

\smallskip\noindent
Suppose instead that $N$ is meager in $H$. Define the coset equivalence relation
\[
 E=\{(x,y)\in H\times H:x-y\in N\}.
\]
For the continuous difference map $d\colon H\times H\to H$, $d(x,y)=x-y$, we have $E=d^{-1}(N)$. Closure of analytic sets under continuous inverse images~\cite[Proposition~14.4]{K1} shows that $E$ is analytic in the Polish space $H\times H$. Thus $E$ has the Baire property in $H \times H$, by~\cite[Theorem~21.6]{K1}. For every $x\in H$, its vertical section is
$E_x=\{y\in H:x-y\in N\}=x+N$, since $N=-N$. Translation by $x$ is a homeomorphism of $H$, so $E_x$ is meager in $H$. Apply the Kuratowski--Ulam theorem~\cite[Theorem~8.41]{K1} to the set $E\subseteq H\times H$ with the Baire property: since every vertical section is meager, in particular the sections are meager for comeager many $x$, and the theorem gives that $E$ is meager in $H\times H$. By Lemma~\ref{mycielski_lemma}, there is a Cantor set $C\subseteq H$ whose distinct elements are pairwise inequivalent under $E$. They represent distinct cosets of $N$, so $[H:N]\geq|C|=\mathfrak{c}$. Together with the upper bound, this gives $[\varphi(A):N]=\mathfrak{c}$.
\end{proof}

\begin{corollary}\label{Polish_invariants} Let $A$ be a Polish $R$-module and let $\grf/\psi$ be a pp-pair in a finite tuple $\bfu$ of free variables. Then the index $[\grf(A) : \psi(A)]$ is either countable or equal to $\mathfrak{c}$. In other words, we have that the following holds:
	\begin{equation}\Th_{\Pun} (A) \vdash [\grf \; \colon \psi] > 1 \;\; \mbox{if and only if} \;\; [\grf(A) : \psi(A)] = \mathfrak{c}.\end{equation}
\end{corollary}

\begin{proof} By Lemma~\ref{pp_analytic}, $\psi(A)$ is an analytic subgroup of $A^{|\bfu|}$, and $\psi(A) \subseteq \grf(A)$, since $\vdash \psi(\bfu) \to \grf(\bfu)$; hence, Theorem~\ref{analytic_dichotomy} gives the asserted dichotomy. By definition, $\Th_{\Pun}(A)$ assigns index $1$ to every pp-pair whose index in $A$ is countable, and infinite index to every other pp-pair. Thus Proposition~\ref{bin_quotient} gives $\Th_{\Pun}(A)\vdash[\grf:\psi]>1$ exactly in the uncountable case, which the dichotomy identifies with index~$\mathfrak{c}$.
\end{proof}

	\begin{fact}[{\cite[14.4]{K1}}]\label{countable_union_analytic} Let $X$ be a Polish space and let $A_n \subseteq X$, for $n < \omega$, be analytic sets. Then $\bigcup_{n<\omega} A_n$ is analytic.
\end{fact}

\begin{corollary}\label{sum_dichotomy} Let $A$ be a Polish $R$-module, let $\varphi(\bfu)$ be a pp-formula and let $(\psi_n(\bfu) : n < \omega)$ be pp-formulae in the same finite tuple of free variables such that $\psi_n(A) \subseteq \varphi(A)$, for every $n < \omega$. Then the index $[\varphi(A) : \Sigma_{n < \omega} \, \psi_n(A)]$ is either countable or equal to $\mathfrak{c}$. In particular, the conclusion of Theorem~\ref{Thm_Second_qf} can be strengthened from ``the index $[\varphi (A) \colon \Sigma_i \, \psi_i (A)]$ is uncountable'' to ``the index $[\varphi (A) \colon \Sigma_i \, \psi_i (A)]$ has the cardinality of the continuum''.
\end{corollary}

\begin{proof} Consider the sum $\Sigma_i \, \psi_i (A)$ as the union of the ascending chain of pp-definable, and so analytic, subgroups $\psi'_n (A) = \sum_{i \leq n} \psi_i (A)$ to see, by Fact~\ref{countable_union_analytic}, that it is analytic. Then apply Theorem~\ref{analytic_dichotomy}.
\end{proof}

\section{The DCC criterion}

\subsection{Kaplansky decompositions} 

As mentioned in the introduction, a direct sum decomposition $K = \bigoplus_{i \in I} \; K_i$ will be called a \textbf{Kaplansky decomposition} if each of the factors $K_i$ is countably generated\footnote{Kaplansky's original result, that projective modules are direct sums of countably generated projectives, doesn't depend on the cardinality of the ring, so this definition is faithful to the spirit of his original result. For $R$ countable, it is the same as that in the introduction.}. As we are assuming that the ring $R$ is countable, this implies that the $K_i$ are countable.

 \medskip

\begin{proposition}[Sabbagh's Theorem; see \cite{Sabbagh}]\label{Sabbagh}
	An embedding $i \colon M \hookrightarrow N$ of modules is elementary if and only if it is pure and $M \equiv N$.
\end{proposition}

We write $M\preccurlyeq N$ to mean that $M$ is an elementary substructure of $N$ in the language of modules. Recall that the notion of bin was introduced in Definition~\ref{the bin}.

\begin{lemma} \label{Kap_bin}
	Let $K$ be a module with a Kaplansky decomposition 
	$K = \bigoplus_{i \in I} \; K_i$. If $M \preccurlyeq K$ is countably generated, and $J \subseteq I$ is a countable subset such that $M \subseteq \bigoplus_{j \in J} \; K_j$, then $M \preccurlyeq \bigoplus_{j \in J} \; K_j \preccurlyeq \bigoplus_{i \in I} \; K_i = K$. Thus any countable direct summand of $K$ that contains a bin $B \preccurlyeq K$ is itself a bin, a \textbf{summand bin}.
\end{lemma}

\begin{proof}
Put $K_J=\bigoplus_{j\in J}K_j$. Since $M$ is pure in $K$ and $K_J$ is a direct summand of $K$, both inclusions $M\subseteq K_J\subseteq K$ are pure. Thus, for every pp-pair $F$, they induce injections $F(M)\hookrightarrow F(K_J)\hookrightarrow F(K)$, by~\eqref{pure_pp_pair_exact}. As $M\preccurlyeq K$, the outer groups have the same finite cardinality or are both infinite; the intermediate group has the same finite cardinality in the first case and is infinite in the second. Baur--Monk therefore gives $M\equiv K_J\equiv K$, and Sabbagh's Theorem makes both pure inclusions elementary. Finally, if $M$ is a bin and $F(K)$ is countable, the composite $F(M)\to F(K)$ is onto, so $F(K_J)\to F(K)$ is onto as well. This proves the bin assertion.
\end{proof}

\subsection{The proof of the DCC criterion}
In this section we prove Theorem~\ref{Thm_DCC}. We first obtain uncountably many independent countable summands with a common theory. If this theory has a strict descending pp-chain, we choose elements witnessing strictness in these orthogonal summands and arrange that suitable infinite sums, together with solutions of the equations defining the pp-formulae, converge in $A$. The resulting elements violate the finite-support requirement of the given direct-sum decomposition, by a $\Delta$-system argument.

\begin{notation} Recalling Definition~\ref{def_pp_formula}, in the case the pp-formula $\varphi(x)$ is a formula in one free variable $x$ (in the language of $R$-modules), then the formula has the form:
	\begin{equation}\exists y_0, ..., y_{j_*-1} (\bigwedge_{i < i_*} r_{i} x = \sum \{s_{(i, j)} y_{j} : j < j_i\}),\end{equation}
	with $r_i, s_{(i, j)} \in R$. Notice that w.l.o.g. we can assume that for all $i < i_*$ we have that $j_i = j_*$, since we can let $j_* = \mrm{max}(\{j_i : i < i_*\})$ and let $s_{(i, j)} = 0$ when undefined.  This notation will be used throughout this section.
\end{notation}

	\begin{observation}\label{observation_prelim} Suppose that $(G, d)$ is a Polish metric abelian group whose metric $d$ is translation-invariant, $A \subseteq G^k$ is uncountable, $1 \leq k < \omega$ and $\zeta > 0$. Then for some $(a_{(1, \ell)} : \ell < k) = \mathbf{a}_1 \neq \mathbf{a}_2 = (a_{(2, \ell)} : \ell < k) \in A$ we have that $d(a_{(1, \ell)} - a_{(2, \ell}), 0_G) < \zeta$, for every $\ell < k$. 
\end{observation}

	\begin{theorem}\label{first_theorem} If (A) then (B), where:
	\begin{enumerate}[(A)]
	\item 
	\begin{enumerate}[(1)]
	\item $R$ is countable;
	\item $A$ is an $R$-module;
		\item $(A,+)$ is a Polish abelian group and $d$ is a compatible complete metric; moreover, $d$ is translation-invariant, meaning that $d(x+a,y+a)=d(x,y)$ for all $a,x,y\in A$;
	\item $A = \bigoplus \{A_\alpha : \alpha < \beta \}$ for $\omega_1 \leq \beta$ and, for every $\alpha < \beta$, $A_\alpha$ is nonzero and countable;
	\item if $r \in R$ and $(x_n : n < \omega)$ and $(rx_n : n < \omega)$ are Cauchy sequences (with respect to the metric $d$ from (3)) with limit $x$ and $y$, \mbox{respectively, then $rx = y$.}
	\end{enumerate}
	\item There are $B_1$, $B_2$, $S_1$, $S_2$ such that:
	\begin{enumerate}[(a)]
	\item $S_1, S_2$ form a partition of $\beta$;
	\item $B_\ell = \bigoplus \{A_\alpha : \alpha \in S_\ell\}$, for $\ell =1, 2$ (so $A =B_1 \oplus B_2$);
	\item $S_1$ is countable (hence $B_1$ is countable);
	\item $B_2$ is a $\Sigma$-algebraically compact $R$-module (hence, for every $\alpha \in S_2$, $A_\alpha$ is $\Sigma$-compact);
\item $B_1$ is a bin of $A$, and $B_2\models\Th_{\Pun}(A)$.
		\end{enumerate}
	\end{enumerate}
\end{theorem}

Notice that a metric as in Theorem~\ref{first_theorem} may always be chosen on an abelian Polish group; see \cite[Exercise~2.2.9]{gao}. In the application to a Polish $R$-module, condition~(A)(5) follows from the continuity of scalar multiplication.

\begin{proof}
By Lemma~\ref{Kap_bin}, choose a countable set $S_1\subseteq\grb$ such that $B_1=\bigoplus_{\gra\in S_1}A_{\gra}$ is a bin of $A$. Put $S_2=\grb\setminus S_1$, $B_2=\bigoplus_{\gra\in S_2}A_{\gra}$ and $T=\Th_{\Pun}(A)$. Proposition~\ref{bin_quotient} gives $B_2\models T$. More precisely, for each pp-pair $F$, the identity $F(A)=F(B_1)\oplus F(B_2)$ and the bin condition show that $F(B_2)$ is zero when $F(A)$ is countable, and uncountable otherwise.

\smallskip \noindent
We construct pairwise disjoint countable sets $W_{\grg}\subseteq S_2$ by recursion on $\grg<\gro_1$. Suppose that the sets $W_{\grd}$ have been chosen for $\grd<\grg$. Let $S_{(2,\grg)}=\bigcup_{\grd<\grg}W_{\grd}$, with $S_{(2,0)}=\varnothing$, and let $C_{\grg}=\bigoplus_{\gra\in S_2\setminus S_{(2,\grg)}}A_{\gra}$ be the remaining direct summand. Since $\grg$ is countable, the part removed from $B_2$ is countable. For every pp-pair $F$, removing that part preserves the alternatives $F(B_2)=0$ and $F(B_2)$ uncountable, so $C_{\grg}\models T$. By L\"owenheim--Skolem and Lemma~\ref{Kap_bin}, choose a countable set $W_{\grg}\subseteq S_2\setminus S_{(2,\grg)}$ such that the module $K_{\grg}:=\bigoplus_{\gra\in W_{\grg}}A_{\gra}$ is an elementary submodule of $C_{\grg}$. In particular, $K_{\grg}\models T$.

\smallskip \noindent
Put $I_*:=\grb\setminus\bigl(S_1\cup\bigcup_{\grg<\gro_1}W_{\grg}\bigr)$. We have the disjoint union
\begin{equation}\label{partition}
\grb=S_1\cupdot\bigcup^{\bullet}_{\grg<\gro_1}W_{\grg}\cupdot I_*,
\end{equation}
and the corresponding coarser decomposition
\begin{equation*}
A=\bigoplus_{\gra<\grb}A_{\gra}=B_1\oplus\bigoplus_{\grg<\gro_1}K_{\grg}\oplus\bigoplus_{\grd\in I_*}A_{\grd}.
\end{equation*}
Here $B_1$ is a summand bin, every $K_{\grg}$ is a countable model of $T$, and $I_*$ indexes the unused original summands.

\smallskip \noindent
It suffices to prove that $T$ is totally transcendental, since $B_2\models T$. The important point is that $B_2\equiv K_{\grg}$ for every $\grg<\gro_1$. For any pp-pair $\grf/\psi$, strictness of $\psi(M)\subseteq\grf(M)$ is expressed by the first-order sentence $[\grf:\psi]>1$. Thus a strict descending pp-chain in $B_2$ would be strict in every $K_{\grg}$. Ruling out such a common chain therefore proves the DCC for the whole of $B_2$, including the summands indexed by $I_*$.

\smallskip \noindent
Suppose, towards a contradiction, that $T$ is not totally transcendental. We can then find a descending chain of pp-formulae $\grf_n (x)$, starting with $\grf_0(x)\equiv x=x$, such that for every summand $K_{\grg}$, $\grg < \gro_1$, the descending chain of pp-definable subgroups
$$K_{\grg} = \grf_0 (K_{\grg})\supset \grf_1 (K_{\grg}) \supset \cdots \supset \grf_n (K_{\grg}) \supset \cdots $$ is strict. Let us write
$\varphi_n (x) \, \dot{=} \, \exists \mathbf{y}_n \, \psi_n(x, \mathbf{y}_n)$, $\mathbf{y}_n = (y_0, ..., y_{j_n-1})$ and 
\begin{equation} \label{new_pp_form} \psi_n(x, \mathbf{y}_n) \, \dot{=}\, \bigwedge_{i < i_n} r_{i} x = \sum \{s_{(i, j)} y_{j} : j < j_i\},\end{equation}
where $i_n < i_{n+1}$ and $j_i \leq j_{i+1}$. Expressing the pp-formulae in this way ensures that 
\begin{equation} \label{pp-implication}
	{_R}T \vdash \grf_{k+1} (x) \to \grf_k (x),
\end{equation}
so that the chain is (properly) descending even for $A$,
$$A = \grf_0 (A)\supset \grf_1 (A) \supset \cdots \supset \grf_n (A) \supset \cdots .$$	

\bigskip \noindent Recall now that the metric $d$ from item (A)(3) of the statement of the theorem induces a norm on $A$ by defining, for $a \in A$, $\| a\| = d(a, 0_A)$.
	\begin{enumerate}[$(*_1)$] 
	\item For $n < \omega$ and $\grg < \gro_1$, we can find $x_{(\grg, n)}$, $y_{(\grg, n, j)}$, for $j < j_n$ such that the following happens:
	\begin{enumerate}[$(\cdot_1)$]
	\item $\mathbf{y}_{(\grg, n)} = (y_{(\grg, n, j)} : j < j_n)$, where $y_{(\grg, n, j)} \in K_\grg$; 
	\item $K_\grg \models \psi_n(x_{(\grg, n)}, \mathbf{y}_{(\grg, n)})$, 
	that is (recalling (\ref{new_pp_form})), we have 
		\begin{equation}
			r_ix_{(\grg,n)}=\sum_{j<j_i}s_{(i,j)}y_{(\grg,n,j)} \qquad (i<i_n);
		\end{equation}
	\item recall $\varphi_n \, \dot{=} \, \exists \mathbf{y}_n \psi_n(x, \mathbf{y}_n)$		
	\item $x_{(\grg, n)} \in \varphi_{n}(K_{\grg}) \setminus \varphi_{n+1}(K_{\grg})$.
	\end{enumerate}
\end{enumerate}
\begin{enumerate}[$(*_2)$] 
	\item Using Observation~\ref{observation_prelim}, by induction on $n < \omega$, we can choose $(\alpha_n, \beta_n)$ s.t.:
	\begin{enumerate}[(a)]
	\item $\alpha_n \neq \beta_n \in \gro_1 \setminus \{\alpha_\ell, \beta_\ell : \ell < n\}$;
	\item $\| x_{(\alpha_n, n)} - x_{(\beta_n, n)}\| < \frac{1}{5^n}$;
	\item for $i < i_n$ and $j < j_i$ we have that $\| y_{(\alpha_n, n, j)} - y_{(\beta_n, n, j)}\| < \frac{1}{5^n}$;
	\item for $i < i_n$, $\| r_i x_{(\alpha_n, n)} - r_i x_{(\beta_n, n)}\| < \frac{1}{5^n}$;
		\item for $i < i_n$ and $j < j_i$, $\| s_{(i, j)} y_{(\alpha_n, n, j)} - s_{(i, j)} y_{(\beta_n, n, j)}\| < \frac{1}{5^n}$.
	\end{enumerate}
\end{enumerate}
\begin{enumerate}[$(*_3)$] 
	\item For $\eta \in 2^\omega$, $n < \omega$, $i < i_n$ and $j < j_i$ let
	\begin{equation*}(a) \;\; x_{\eta \restriction n} = \sum \{x_{(\alpha_\ell, \ell)} - x_{(\beta_\ell, \ell)} : \ell < n \text{ and } \eta(\ell) = 1\},\end{equation*}
	\begin{equation*}(b) \;\; y_{(\eta \restriction n, j)} = \sum \{y_{(\alpha_\ell, \ell, j)} - y_{(\beta_\ell, \ell, j)} : \ell < n, \eta(\ell) = 1, \text{ and } j_\ell > j\}.\end{equation*}
\end{enumerate}
\begin{enumerate}[$(*_4)$] 
	\item If $\eta\in2^\omega$ and $m \leq n < \omega$, then
	\begin{equation*}\| x_{\eta \restriction n} - x_{\eta \restriction m}\| \leq \sum \{ \frac{1}{5^t} : t \in [m, n)\} < \frac{2}{2^m}.\end{equation*}
\end{enumerate}
\begin{enumerate}[$(*_5)$] 
	\item If $\eta\in2^\omega$, $m\leq n<\omega$, and $i<i_m$, then
	\begin{equation*}\| r_i x_{\eta \restriction n} - r_i x_{\eta \restriction m} \| \leq \sum \{ \frac{1}{5^t} : t \in [m, n)\} < \frac{2}{2^m}.\end{equation*}
\end{enumerate}
\begin{enumerate}[$(*_6)$] 
	\item If $\eta\in2^\omega$, $m\leq n<\omega$, $i<i_m$, and $j<j_i$, then
	\begin{equation*}\| y_{(\eta \restriction n, j)} - y_{(\eta \restriction m, j)}\| \leq \sum \{ \frac{1}{5^t} : t \in [m, n)\} < \frac{2}{2^m}.\end{equation*}
\end{enumerate}
\begin{enumerate}[$(*_7)$] 
	\item If $\eta\in2^\omega$, $m\leq n<\omega$, $i<i_m$, and $j<j_i$, then
	\begin{equation*}\| s_{(i, j)} y_{(\eta \restriction n, j)} - s_{(i, j)}y_{(\eta \restriction m, j)} \| \leq \sum \{ \frac{1}{5^t} : t \in [m, n)\} < \frac{2}{2^m}.\end{equation*}
\end{enumerate}
\begin{enumerate}[$(*_8)$] 
	\item For $\eta \in 2^\omega$ we have that
	\begin{enumerate}[(a)]
	\item 
	\begin{enumerate}[$(\cdot_1)$]
	\item $(x_{\eta \restriction n} : n < \omega)$ is a Cauchy sequence;
	\item let $x_\eta$ be its limit.
	\end{enumerate}
	\item if $m < \omega$, $i < i_m$, and $j < j_i$, then
	\begin{enumerate}[$(\cdot_1)$]
	\item $(y_{(\eta \restriction n, j)} : n \in [m,\omega))$ is a Cauchy sequence;
	\item let $y_{(\eta, j)}$ be its limit.
	\end{enumerate}
	\item if $m<\omega$ and $i<i_m$, then
	\begin{enumerate}[$(\cdot_1)$]
	\item $(r_i x_{\eta \restriction n} : n\in[m,\omega))$ is a Cauchy sequence;
	\item its limit is $r_i x_\eta$, recalling hypothesis (A)(5);
		\end{enumerate}
	\item if $m<\omega$, $i < i_m$, and $j < j_i$, then 
	\begin{enumerate}[$(\cdot_1)$]
	\item $(s_{(i, j)} y_{(\eta \restriction n, j)} : n \in [m, \omega))$ is a Cauchy sequence;
	\item its limit is $s_{(i, j)} y_{({\eta}, j)}$.
	\end{enumerate}
	\end{enumerate}
\end{enumerate}
\begin{enumerate}[$(*_9)$]
	\item If $\eta \in 2^\omega$ and $n < \omega$, then:
	\begin{enumerate}[$(\cdot_1)$]
	\item $x_\eta - x_{\eta \restriction n}$ is the limit of the Cauchy sequence $(x_{\eta \restriction k} - x_{\eta \restriction n} : k \in [n, \omega))$;
	\item $x_{\eta} - x_{\eta \restriction n} \in \varphi_n(A)$.
	\end{enumerate}
\end{enumerate}
Why $(*_{9})$? Concerning $(\cdot_1)$,  $(x_{\eta \restriction n+k} : k < \omega)$ is a Cauchy sequence with limit $x_\eta$, hence $(x_{\eta \restriction n+k} - x_{\eta \restriction n} : k < \omega)$ is a Cauchy sequence with limit $x_{\eta} - x_{\eta \restriction n}$.
\newline For $(\cdot_2)$, fix $i<i_n$. For every $k_*<\omega$, the definitions in $(*_3)$ and the equations in $(*_1)$ give
\begin{equation*}
r_i\bigl(x_{\eta\restriction(n+k_*)}-x_{\eta\restriction n}\bigr)=\sum_{j<j_i}s_{(i,j)}\bigl(y_{(\eta\restriction(n+k_*),j)}-y_{(\eta\restriction n,j)}\bigr).
\end{equation*}
By $(*_8)$ and hypothesis~(A)(5), passing to the limit as $k_*\to\infty$ yields
\begin{equation*}
r_i\bigl(x_\eta-x_{\eta\restriction n}\bigr)=\sum_{j<j_i}s_{(i,j)}\bigl(y_{(\eta,j)}-y_{(\eta\restriction n,j)}\bigr) \qquad (i<i_n).
\end{equation*}
For each coordinate $j<j_n$ which occurs in one of these equations, put $w_{(\eta,n,j)}:=y_{(\eta,j)}-y_{(\eta\restriction n,j)}$, and put $w_{(\eta,n,j)}:=0$ for every unused coordinate. The preceding identities say precisely that $A\models\psi_n(x_\eta-x_{\eta\restriction n},(w_{(\eta,n,j)}:j<j_n))$. Hence $x_\eta-x_{\eta\restriction n}\in\varphi_n(A)$, as required.
\begin{enumerate}[$(*_{10})$]
	\item For $\eta \in 2^\omega$, let $(z_{(\eta, \gamma)} : \gamma \in u_\eta)$ be such that:
	\begin{enumerate}[$(\cdot_1)$]
	\item $x_\eta = \sum \{z_{(\eta, \gamma)} : \gamma \in u_\eta\}$;
	\item $u_\eta \subseteq \beta$ is finite;
	\item $z_{(\eta, \gamma)} \in A_\gamma \setminus \{0\}$, for $\gamma < \beta$;
	\item let $(\gamma_{(\eta, \ell)} : \ell < \mathfrak{n}_\eta)$ list $u_\eta$ with no repetitions.
	\end{enumerate}
\end{enumerate}
Notice that the map $\eta\mapsto x_\eta$ is injective. Indeed, let $k$ be the first coordinate at which distinct $\eta,\nu\in2^\omega$ differ, and, after interchanging them if necessary, assume that $\eta(k)=0$ and $\nu(k)=1$. Put $d_k:=x_{(\alpha_k,k)}-x_{(\beta_k,k)}$. Then
\begin{equation*}
x_\nu-x_\eta=(x_\nu-x_{\nu\restriction(k+1)})+d_k-(x_\eta-x_{\eta\restriction(k+1)}).
\end{equation*}
The two tail terms belong to $\varphi_{k+1}(A)$ by $(*_9)(\cdot_2)$, whereas $d_k\in\varphi_k(A)\setminus\varphi_{k+1}(A)$ by $(*_1)$ and the direct-sum decomposition. It follows that $x_\nu-x_\eta\in\varphi_k(A)\setminus\varphi_{k+1}(A)$, and in particular $x_\nu\neq x_\eta$.

\begin{enumerate}[$(*_{11})$]
	\item There are $\mathfrak{m}_* \leq \mathfrak{n}_*$ and $S \subseteq 2^\omega$ of cardinality $\aleph_1$ such that:
	\begin{enumerate}[(a)]
	\item $(u_\eta : \eta \in S)$ is a $\Delta$-system with heart $u_*$;
	\item $|u_*| = \mathfrak{m}_*$ and $|u_\eta| = \mathfrak{n}_*$;
	\item $u'_\eta = u_\eta \setminus u_*$ is disjoint from $\{\alpha_\ell, \beta_\ell : \ell < \omega\}$.
	\end{enumerate}
\end{enumerate}
Why $(*_{11})$? As there are only $\aleph_0$-many possibilities for $|u_\eta|$, w.l.o.g. there is an uncountable $S \subseteq 2^\omega$ and $\mathfrak{n}_* < \omega$ such that $\eta \in S$ implies $|u_{\eta}| = \mathfrak{n}_*$. Then using the $\Delta$-system lemma, we immediately get clause (a), and recalling the previous sentence, this also gives clause (b). As the $(u_\eta \setminus u_* : \eta \in S)$ are pairwise disjoint, $S$ is uncountable and $\{\alpha_\ell, \beta_\ell : \ell < \omega\}$ is countable, w.l.o.g., clause (c) holds as well.

\begin{enumerate}[$(*_{12})$]
	\item W.l.o.g. for every $\eta \in S$, we have:
	\begin{enumerate}[$(\cdot_1)$]
	\item $\{ \gamma_{(\eta, \ell)}  :\ell < \mathfrak{m}_* \}$ is $u_*$;
	\item $\gamma_{(\eta, \ell)} = \gamma^*_\ell$, for $\ell < \mathfrak{m}_*$.
\end{enumerate}
\end{enumerate}
\begin{enumerate}[$(*_{13})$]
	\item W.l.o.g. $(z_{(\eta, \gamma^*_\ell)} : \ell < \mathfrak{m}_*)$ is constant, say equal to $(z^*_\ell : \ell < \mathfrak{m}_*)$, hence necessarily $\mathfrak{m}_* < \mathfrak{n}_*$.
\end{enumerate}
[Why w.l.o.g. $(z_{(\eta, \gamma^*_\ell)} : \ell < \mathfrak{m}_*)$ is constant? As each $A_\gamma$ is countable. Why do we have $\mathfrak{m}_* < \mathfrak{n}_*$? Otherwise every support is the common heart, and the thinning in $(*_{13})$ makes $(x_\eta : \eta\in S)$ constant, contradicting the injectivity proved above.]
\begin{enumerate}[$(*_{14})$]
	\item 
	\begin{enumerate}[$(\cdot_1)$]
	\item Let $z_* = \sum \{z^*_\ell : \ell < \mathfrak{m}_* \}$;
	\item for $\eta \in S$, let $x'_\eta = x_\eta - z_*$;
	\item $x'_\eta \neq 0$.
	\end{enumerate}
\end{enumerate}
[Why $(\cdot_3)$? Otherwise, recalling $(*_{10})(\cdot_3)$, we have that $\mathfrak{n}_* = \mathfrak{m}_*$.]
\begin{enumerate}[$(*_{15})$]
	\item For $\eta \in S$, let $\mathfrak{k}(\eta) \leq \omega$ be such that:
	\begin{enumerate}[$(\cdot_1)$]
	\item $\mathfrak{k}(\eta) = \omega$ iff $x'_\eta \in \bigcap_{n < \omega} \varphi_{n}(A)$;
	\item $\mathfrak{k}(\eta) = n < \omega$ iff there is $n$ (in which case $n$ is unique) such that 
	\begin{equation}x'_\eta \in \varphi_{n}(A) \setminus \varphi_{n+1}(A).\end{equation}
	\end{enumerate}
\end{enumerate}
\begin{enumerate}[$(*_{16})$]
	\item Notice that  $\mathfrak{k}(\eta)$ is well-defined because by assumption $\varphi_0(A) = A$.
\end{enumerate}
\begin{enumerate}[$(*_{17})$]
	\item W.l.o.g., for some $\mathfrak{k}_* \leq \omega$ we have that $\eta \in S$ implies $\mathfrak{k}(\eta) = \mathfrak{k}_*$.
\end{enumerate}
[Why? $S$ is uncountable,  $\mathfrak{k}(\eta)$ has $|\omega+1| =\aleph_0$ many possible values, hence at least one value is chosen $\geq \aleph_1$-many times. Thus, we are fine.]
\begin{enumerate}[$(*_{18})$]
	\item Choose $\eta \neq \nu \in S$ and $k < \omega$ such that
	\begin{enumerate}[$(\cdot_1)$]
	\item $\eta \restriction k = \nu \restriction k$, $\eta(k) = 0$, $\nu(k) = 1$;
	\item $\mathfrak{k}_* < \omega$ implies $k > \mathfrak{k}_*$;
\end{enumerate}	
\end{enumerate}
\begin{enumerate}[$(*_{19})$]
	\item Put $d_k:=x_{(\alpha_k,k)}-x_{(\beta_k,k)}$. Since $\eta\restriction k=\nu\restriction k$, $\eta(k)=0$, and $\nu(k)=1$, we have $x_{\nu\restriction(k+1)}=x_{\eta\restriction(k+1)}+d_k$, and hence
\begin{equation*}
x'_\nu-x'_\eta=x_\nu-x_\eta=(x_\nu-x_{\nu\restriction(k+1)})+d_k-(x_\eta-x_{\eta\restriction(k+1)}).
\end{equation*}
\end{enumerate}
 Now we have 
\begin{enumerate}[$(*_{20})$]
	\item 
	\begin{enumerate}[$(\cdot_1)$]
	\item $x_{(\alpha_k, k)} - x_{(\beta_k, k)} \in \varphi_{k}(A) \setminus \varphi_{k+1}(A)$;
	\item $(x_\nu - x_{\nu \restriction (k+1)}) \in \varphi_{k+1}(A)$ (by $(*_{9})(\cdot_2)$);
	\item $(x_\eta - x_{\eta \restriction (k+1)}) \in \varphi_{k+1}(A)$ (by $(*_{9})(\cdot_2)$).
\end{enumerate}	
\end{enumerate}
[Why $(*_{20})$? We give details only on $(\cdot_1)$. The fact that $x_{(\alpha_k, k)} - x_{(\beta_k, k)} \in \varphi_k(A)$ is immediate. The fact that $x_{(\alpha_k, k)} - x_{(\beta_k, k)} \notin \varphi_{k+1}(A)$ follows from the fact that by assumption $x_{(\alpha_k, k)} \in K_{\alpha_k}$, $x_{(\beta_k, k)} \in K_{\beta_k}$ and $K_{\alpha_k} + K_{\beta_k} = K_{\alpha_k} \oplus K_{\beta_k}$.]

\smallskip \noindent
Now, as $\varphi_{k+1} \vdash \varphi_{k}$ (\ref{pp-implication}), by $(*_{19})$ and $(*_{20})$, we have that:
\begin{enumerate}[$(*_{21})$]
	\item $x'_\nu - x'_\eta \in \varphi_{k}(A) \setminus \varphi_{k+1}(A)$.
\end{enumerate}
Notice that it follows from $(*_{21})$ that $\mathfrak{k}_* < \omega$; in fact if not then, by $(*_{17})$, $\omega = \mathfrak{k}_* = \mathfrak{k}(\eta) = \mathfrak{k}(\nu)$ and so $x'_\eta, x'_\nu \in \bigcap_{n < \omega} \varphi_n(A)$, which implies that $x'_\eta - x'_\nu \in \bigcap_{n < \omega} \varphi_n(A)$, contradicting $(*_{21})$ (since obviously $k+1 < \omega$).
\begin{enumerate}[$(*_{22})$]
	\item We have that
\begin{equation}x'_\eta = \sum \{z_{(\eta, \gamma)} : \gamma \in u_\eta \setminus u_*\} \text{ and } x'_\nu = \sum \{z_{(\nu, \gamma)} :\gamma \in u_\nu \setminus u_* \}.\end{equation}
\end{enumerate}
[Why? By $(*_{10})(\cdot_1)$, $(*_{13})$ and $(*_{14})(\cdot_2)$.]
\begin{enumerate}[$(*_{23})$]
	\item $x'_\nu, x'_\eta \in \varphi_{\mathfrak{k}_*}(A) \setminus \varphi_{\mathfrak{k}_*+1}(A)$.
\end{enumerate}
[Why? By $(*_{15})(\cdot_2)$ and $(*_{17})$.]
\newline By $(*_{22})$ and $(*_{23})$ recalling that $(u_\eta \setminus u_*) \cap (u_\nu \setminus u_*) = \emptyset$ (see $(*_{11})$(a)) we have
\begin{enumerate}[$(*_{24})$]
\item $x'_\nu - x'_\eta \in \varphi_{\mathfrak{k}_*}(A) \setminus \varphi_{\mathfrak{k}_*+1}(A)$.
\end{enumerate}
Recall that $\mathfrak{k}_* < k$ (cf. $(*_{18})(\cdot_2)$); thus $(*_{21})$ and $(*_{24})$ yield a contradiction.
\end{proof}

\section{The main theorem}\label{sec_main_theorem}

Let us combine the results of the previous two sections to both characterize when an uncountable Polish module has a Kaplansky decomposition as well as to determine the isomorphism type of the unlimited complement of a summand bin. Recall from Subsection~\ref{sec_ziegler} that $\Cl(U)=\Cl(\Th(U))$ is the closed subset of $\Zg(R)$ associated with $U$, and that $M_{\cC}=\bigoplus_{V\in\cC}V$ denotes the direct sum of one representative of each point of a closed set $\cC$.

\begin{theorem}[Main Theorem]\label{main_th_sec} 
	Let $A$ be an $R$-module with a Kaplansky decomposition \begin{equation}A = \bigoplus_{i \in I} A_i,\end{equation} where $A_i$ is a countable $R$-module, for all $i \in I$. Then $A$ admits a Polish $R$-module structure if and only if there is a partition $I = I_1 \cupdot I_2$ and a totally transcendental closed set $\cC \subseteq \Zg (R)$ such that $A = B \oplus U$, where
	\begin{enumerate}[(1)]
		\item $B = \bigoplus_{i \in I_1} A_i$ is a bin of $A$ (so that $I_1$ is necessarily countable) and 
		\item $U = \bigoplus_{i \in I_2} A_i \cong (M_{\cC})^{(\mathfrak{c})} \isom (M_{\cC})^{\gro}$ (so that $U$ is necessarily $\Sigma$-algebraically compact with $\Cl (U) = \cC$).
	\end{enumerate}

\end{theorem}

\begin{proof}
If $A$ is countable, take $I_1=I$, $I_2=\varnothing$, $B=A$, $U=0$, and $\cC=\varnothing$; the discrete topology gives the assertion. Here $I$ is countable by the standing nonzero-summand convention. We may therefore assume that $A$ is uncountable.

\smallskip \noindent
For sufficiency, Lemma~\ref{tt_closed_set} shows that $M_{\cC}$ is countable, so $U\cong(M_{\cC})^{\gro}$ admits a Polish $R$-module topology as a countable product of discrete countable modules. Giving $B$ the discrete topology then makes $A=B\oplus U$ Polish. This proves sufficiency.

\smallskip \noindent
For necessity, assume that $A$ is Polish. Theorem~\ref{first_theorem} gives that $A=B\oplus U$, where $B$ is a summand bin and $U$ is $\Sigma$-algebraically compact with $\Th(U)=T:=\Th_{\Pun}(A)$. Thus $\cC:=\Cl(U)=\Cl(T)$ is a totally transcendental closed set.

\smallskip \noindent
By Garavaglia's theorem~\cite{garavaglia}, write $U=\bigoplus_{V\in\mathcal D}V^{(\grl_V)}$, where $\mathcal D$ consists of the pairwise nonisomorphic indecomposable pure-injectives that actually occur in $U$, and $\grl_V>0$. A pp-pair $F=\grf/\psi$ is said to vanish on a module $N$ if $F(N)=0$, equivalently $\grf(N)=\psi(N)$. Every pp-pair vanishing on $U$ also vanishes on each direct summand $V$, so $\mathcal D\subseteq\cC$ by the definition of $\Cl(U)$. We first prove $\mathcal D=\cC$, and then determine the cardinals $\grl_W$ indexing the copies of each $W\in\cC$.

\smallskip \noindent
To prove $\cC\subseteq\mathcal D$, fix $W\in\cC$. Recall that ${_R}T$ denotes the axioms for $R$-modules, so ${_R}T\vdash\psi\to\grf$ means that the implication holds in every $R$-module. Here $T=\Th(U)=\Th_{\Pun}(A)$ is the complete theory fixed above. The DCC modulo $T$ gives a $T$-minimal pp-formula $\grf_W(u)$ among those with $\grf_W(W)\neq0$. Thus, if ${_R}T\vdash\psi\to\grf_W$, either $\psi\equiv\grf_W$ modulo $T$ or $\psi(W)=0$. Recall also that every pp-implication valid in $T$ holds at every point of $\cC=\Cl(T)$, by the definition of~$\Cl(T)$.

\smallskip \noindent
Enumerate, with repetitions if necessary, all pp-formulae $\psi(u)$ such that ${_R}T\vdash\psi\to\grf_W$ and $\psi(W)=0$ as $(\psi_i(u):i<\omega)$. For every $n<\omega$, put $\theta_n=\sum_{i\leq n}\psi_i$. Then $\theta_n\leq\theta_{n+1}\leq\grf_W$ and $\theta_n(W)=0$ for every $n<\omega$. With $W$ and $\grf_W$ fixed, define, for every module $N$, the subgroup $H_{\grf_W}(N)=\bigcup_{n<\omega}\theta_n(N)=\sum_{i<\omega}\psi_i(N)$. In particular, $H_{\grf_W}(N)\subseteq\grf_W(N)$ and $H_{\grf_W}(W)=0$.

\smallskip \noindent
We first prove that $\grf_W(V)=H_{\grf_W}(V)$ for every $V\in\cC$ with $V\not\cong W$. Fix $0\neq w\in\grf_W(W)$ and suppose, towards a contradiction, that $a\in\grf_W(V)\setminus H_{\grf_W}(V)$. Given any pp-formula $\rho(u)$, apply minimality to $\grf_W\wedge\rho$. If $\grf_W\wedge\rho\equiv\grf_W$ modulo $T$, the pp-pair $\grf_W/(\grf_W\wedge\rho)$ vanishes in $T$ and hence on both $W$ and $V$, because these belong to $\Cl(T)$. Consequently, $\grf_W(W)=(\grf_W\wedge\rho)(W)$ and $\grf_W(V)=(\grf_W\wedge\rho)(V)$. Since $w\in\grf_W(W)$ and $a\in\grf_W(V)$, both satisfy $\rho$.

\smallskip \noindent
In the other case, $(\grf_W\wedge\rho)(W)=0$. The nonzero element $w\in\grf_W(W)$ cannot satisfy $\rho$, since it would then belong to this zero subgroup. Moreover, $\grf_W\wedge\rho=\psi_j$ for some $j<\omega$, by our enumeration. Hence $(\grf_W\wedge\rho)(V)=\psi_j(V)\subseteq\theta_j(V)\subseteq H_{\grf_W}(V)$, where $H_{\grf_W}(V)=\bigcup_{n<\omega}\theta_n(V)$. If $a$ satisfied $\rho$, its membership in $\grf_W(V)$ would place it in $H_{\grf_W}(V)$, contrary to the choice of $a$. Thus neither $w$ nor $a$ satisfies $\rho$ in this case. We have proved $\tp^+(W,w)=\tp^+(V,a)$. Since $a\notin H_{\grf_W}(V)$ while $0\in H_{\grf_W}(V)$, we have $a\neq0$. Hence both $w$ and $a$ are nonzero, as required in Fact~\ref{pointed_hull_fact} in Subsection~\ref{sec_purity}. Applying that fact to their equal positive pp-types gives $W\cong V$, the required contradiction. Therefore $\grf_W(V)=H_{\grf_W}(V)=\sum_{i<\omega}\psi_i(V)$ for every $V\in\cC\setminus\{W\}$.

\smallskip \noindent
For any direct sum $N=\bigoplus_{j\in J}N_j$, we have $H_{\grf_W}(N)=\bigoplus_{j\in J}H_{\grf_W}(N_j)$ and $\grf_W(N)=\bigoplus_{j\in J}\grf_W(N_j)$.

\smallskip \noindent
We now calculate the index $[\grf_W(A):H_{\grf_W}(A)]$. For every $n<\omega$, we have $\theta_n(W)=0\neq\grf_W(W)$, so $\grf_W/\theta_n$ does not vanish on $W\in\Cl(T)$. It therefore cannot have index $1$ in $T$; its index is infinite, since $T=\Th_{\Pun}(A)$. Every finite part of $\{\grf_W\}\cup\{\neg\theta_n:n<\omega\}$ is consequently consistent with $T$: the finitely many excluded subgroups are all contained in one $\theta_n$, for some $n<\omega$, and this is a proper subgroup of $\grf_W$ in every model of $T$. Theorem~\ref{Thm_Second_qf} gives $[\grf_W(A):H_{\grf_W}(A)]>\aleph_0$. For every $n<\omega$, $\theta_n(A)$ is analytic by Lemma~\ref{pp_analytic}. Its increasing union $H_{\grf_W}(A)=\bigcup_{n<\omega}\theta_n(A)$ is an analytic subgroup by Fact~\ref{countable_union_analytic}. Theorem~\ref{analytic_dichotomy} now gives $[\grf_W(A):H_{\grf_W}(A)]=\mathfrak{c}$.

\smallskip \noindent
We use this index to prove that $W$ actually occurs in $U$. Suppose that $W\notin\mathcal D$. Every $V\in\mathcal D$ then lies in $\cC\setminus\{W\}$, so the pp-type argument above gives $\grf_W(V)=H_{\grf_W}(V)$. Applying the two direct-sum identities to $U=\bigoplus_{V\in\mathcal D}V^{(\grl_V)}$, we obtain $\grf_W(U)=H_{\grf_W}(U)$. Applying them also to $A=B\oplus U$, projection onto $B$ induces $\grf_W(A)/H_{\grf_W}(A)\cong\grf_W(B)/H_{\grf_W}(B)$. Since $B$ is countable, this would give $[\grf_W(A):H_{\grf_W}(A)]\leq\aleph_0$, contradicting the value $\mathfrak{c}$ already proved. Thus $W\in\mathcal D$. This proves $\cC\subseteq\mathcal D$ because $W\in\cC$ was arbitrary. Together with $\mathcal D\subseteq\cC$, we now obtain $\mathcal D=\cC$.

\smallskip \noindent
We now determine the cardinals $\grl_W$, that is, the numbers of copies of each $W\in\cC$ in the decomposition of $U$. Fix $W\in\cC=\mathcal D$ and use the subgroups constructed for this $W$. We have $H_{\grf_W}(W)=0$, while $\grf_W(V)=H_{\grf_W}(V)$ for every $V\in\mathcal D\setminus\{W\}$. The direct-sum identities therefore give
\begin{equation*}
\grf_W(A)/H_{\grf_W}(A)\cong\bigl(\grf_W(B)/H_{\grf_W}(B)\bigr)\oplus\grf_W(W)^{(\grl_W)}.
\end{equation*}
The quotient on the left has cardinality $\mathfrak{c}$. On the right, the first group is countable, and $\grf_W(W)$ is a nonzero subgroup of the countable module $W$ by Lemma~\ref{tt_closed_set}. Hence $\mathfrak{c}\leq\max(\aleph_0,\grl_W)$. Since also $\grl_W\leq|A|=\mathfrak{c}$, we obtain $\grl_W=\mathfrak{c}$. This holds for every $W\in\cC$, and therefore $U\cong(M_{\cC})^{(\mathfrak{c})}$.

\smallskip \noindent
It remains to prove $U\cong(M_{\cC})^{\gro}$. Put $P=(M_{\cC})^{\gro}$. Since $U$ is uncountable, $\cC$ is nonempty and $P$ is uncountable. The module $M_{\cC}$ is countable by Lemma~\ref{tt_closed_set}, so its countable product $P$ is Polish. It is also $\Sigma$-algebraically compact: pp-formulae commute with products, so the DCC passes from $M_{\cC}$ to $P$. The same pp-pairs vanish on $P$ and $M_{\cC}$, giving $\Cl(P)=\cC$. Proposition~\ref{discrete_product} gives $\Th(P)=\Th_{\Pun}(P)$. Thus $\Th(P)$ and $T$ have the same associated closed set and both have only $1$ or infinite pp-indices, so Fact~\ref{Ziegler_topology} gives $\Th(P)=T$.

\smallskip \noindent
By Garavaglia's theorem, write $P=\bigoplus_{V\in\mathcal D_P}V^{(\mu_V)}$, where $\mathcal D_P$ consists of its actual indecomposable pure-injective summand types and $\mu_V>0$. The direct-summand argument used for $U$ gives $\mathcal D_P\subseteq\Cl(P)=\cC$. Since $\Th(P)=T$, we may use the same $\grf_W$, $(\theta_n:n<\omega)$ and $H_{\grf_W}$ constructed above for each $W\in\cC$. Applying Theorems~\ref{Thm_Second_qf} and~\ref{analytic_dichotomy} to $P$ gives $[\grf_W(P):H_{\grf_W}(P)]=\mathfrak{c}$.

\smallskip \noindent If $W\notin\mathcal D_P$, every summand in this decomposition satisfies $\grf_W(V)=H_{\grf_W}(V)$, so the direct-sum identities give $\grf_W(P)=H_{\grf_W}(P)$. The index would then be $1$, a contradiction. Hence $\mathcal D_P=\cC$. For $W\in\cC$, only the copies of $W$ contribute to the quotient, giving $\grf_W(P)/H_{\grf_W}(P)\cong\grf_W(W)^{(\mu_W)}$. The left side has cardinality $\mathfrak{c}$, and $W$ is countable; hence $\mathfrak{c}\leq\max(\aleph_0,\mu_W)$, while $\mu_W\leq|P|=\mathfrak{c}$. Therefore $\mu_W=\mathfrak{c}$ for every $W\in\cC$, and $P\cong(M_{\cC})^{(\mathfrak{c})}\cong U$.

\smallskip \noindent
Finally, $\cC=\Cl(\Th_{\Pun}(A))$ depends only on the abstract module $A$, and the normal form therefore makes the isomorphism type of $U$ an invariant of $A$.
\end{proof}

\begin{corollary} \label{non_Arch}
	If an uncountable Polish module $A$ admits a Kaplansky decomposition, then $A$ admits a non-Archimedean Polish topology.
\end{corollary}

\section{Applications to rings and modules}\label{sec_applications}

We now prove the existence theorems announced in the introduction. By Theorem~\ref{main_th_sec}, an uncountable Polish module $A$ with a Kaplansky decomposition has a decomposition $A=B\oplus U$, where $B$ is a countable summand bin and
\[
 U\cong M_{\cC}^{(\mathfrak{c})}\cong M_{\cC}^{\omega},
 \qquad \cC=\Cl(\Th_{\Pun}(A)).
\]
The closed set $\cC$ is nonempty and totally transcendental, and every point of $\cC$ occurs as a direct summand of $U$. Thus $U$ is determined up to isomorphism by $A$. In particular, if $A$ is projective, then every point of $\cC$ is projective. These conclusions concern the unlimited summand, i.e., the summand $U$, while the countable summand $B$, and hence $\Th(A)$, need not be totally transcendental.

\subsection{Uncountable free Polish modules}

Recall that a ring $R$ is \emph{right coherent} if every finitely generated right ideal is finitely presented. We use the following pp-formula characterization.

\begin{fact}[{\cite[Theorem~14.16]{prest_first_book}}]\label{coherence_pp}
Let $R$ be a ring. Then $R$ is right coherent if and only if $\grf({}_{R}R)$ is a finitely generated right ideal for every pp-formula $\grf(u)$ in one free variable.
\end{fact}

If $I=r_1R+\cdots+r_tR$ is a finitely generated right ideal, write $I|u$ for the pp-formula $\exists v_1,\ldots,v_t\,(u\doteq\sum_{j=1}^t r_jv_j)$. Its value in the left regular module ${}_{R}R$ is $I$.

A \emph{projective cover} of a module $M$ is a surjective homomorphism $p\colon P\to M$, with $P$ projective, such that no proper submodule of $P$ maps onto $M$. Equivalently, $\ker p$ is \emph{superfluous}: if $L\leq P$ and $L+\ker p=P$, then $L=P$. A ring is \emph{left perfect} if every $R$-module has a projective cover. The following equivalent condition describes left perfectness directly in terms of ideals.

\begin{fact}[Bass's Theorem, {\cite[Theorem~P]{Bass}}]\label{bass_perfect}
Let $R$ be a ring. Then $R$ is left perfect if and only if every descending chain $a_0R\supseteq a_1R\supseteq\cdots$ of principal right ideals eventually stabilizes.
\end{fact}

The \emph{Jacobson radical} $J(R)$ is the intersection of all maximal right ideals of $R$ (equivalently, of all maximal left ideals). An element $e\in R$ is an \emph{idempotent} if $e^2=e$. Idempotents $e$ and $f$ are \emph{orthogonal} if $ef=fe=0$, and a nonzero idempotent is \emph{primitive} if it cannot be written as a sum of two nonzero orthogonal idempotents. Over a left perfect ring, $R/J(R)$ is semisimple artinian (a finite product of full matrix rings over division rings), the identity $1_R$ is a finite sum of pairwise orthogonal primitive idempotents, and every projective $R$-module is a direct sum of indecomposable summands of ${}_{R}R$~\cite[Theorem~2.1 and p.~476]{Bass}.

\begin{proposition}\label{useful_prop}
Let $R$ be a countable ring. If an uncountable free $R$-module admits a Polish topology, then ${}_{R}R$ is $\Sigma$-algebraically compact and $R$ is left perfect.
\end{proposition}

\begin{proof}
An uncountable Polish space has cardinality $\mathfrak{c}$, so the free module is $A=({}_{R}R)^{(\mathfrak{c})}$. Apply Theorem~\ref{Thm_DCC} to its decomposition into copies of ${}_{R}R$. After removing countably many copies, the remaining summand is still isomorphic to $({}_{R}R)^{(\mathfrak{c})}$ and is $\Sigma$-algebraically compact. Since pp-formulae commute with direct sums, the DCC on its pp-definable subgroups implies the same DCC on ${}_{R}R$. Thus ${}_{R}R$ is $\Sigma$-algebraically compact. In particular, every descending chain of principal right ideals stabilizes, because $aR$ is defined in ${}_{R}R$ by the pp-formula $a|u$. Bass's Theorem (Fact~\ref{bass_perfect}) therefore gives that $R$ is left perfect.
\end{proof}

\begin{proposition}\label{useful_prop_2}
Let $R$ be a countable ring. If an uncountable free $R$-module admits a Polish topology, then $R$ is right coherent.
\end{proposition}

\begin{proof}
Put $A=({}_{R}R)^{(\mathfrak{c})}$. By Fact~\ref{coherence_pp}, we must show that $\grf({}_{R}R)$ is a finitely generated right ideal for every pp-formula $\grf(u)$ in one free variable. It is a right ideal because right multiplication by any $r\in R$ is an endomorphism of ${}_{R}R$, and homomorphisms preserve pp-formulae. Since $R$ is countable, write $\grf({}_{R}R)=\bigcup_{i<\omega}I_i$ as an increasing union of finitely generated right ideals, and put $\psi_i(u)=(I_i|u)$.

Suppose that every $I_i$ is properly contained in $\grf({}_{R}R)$. Then every finite subset of $q(u)=\{\grf(u)\}\cup\{\neg\psi_i(u):i<\omega\}$ is realized in ${}_{R}R$, and hence in any one of the direct summands of $A$ isomorphic to ${}_{R}R$, with all other coordinates zero. By compactness, $q$ is $\Th(A)$-consistent. Proposition~\ref{discrete_product}, applied to the countable module ${}_{R}R$, says that every pp-pair in $A$ has index $1$ or $\mathfrak{c}$; hence $\Th(A)=\Th_{\Pun}(A)$. Theorem~\ref{Thm_CH} now gives
\[
 [\grf(A):\textstyle\sum_{i<\omega}\psi_i(A)]=\mathfrak{c}.
\]
On the other hand, pp-formulae commute with direct sums, so $\grf(A)=\grf({}_{R}R)^{(\mathfrak{c})}$ and $\psi_i(A)=\psi_i({}_{R}R)^{(\mathfrak{c})}=I_i^{(\mathfrak{c})}$. Each element of $\grf(A)$ has finitely many nonzero coordinates. Each such coordinate belongs to some $I_i$, and, since the ideals form an ascending chain, one common $I_i$ contains them all. Thus $\grf(A)=\bigcup_{i<\omega}\psi_i(A)=\sum_{i<\omega}\psi_i(A)$, making the displayed index $1$, a contradiction. Therefore $\grf({}_{R}R)=I_i$ for some $i<\omega$, as required by Fact~\ref{coherence_pp}.
\end{proof}

We next compute the cardinal quotients in O'Neill's freeness criterion; these are also the quotients considered in Example~\ref{O'Neill}.

\begin{lemma}\label{third_condition_pleonastic}
Let $R$ be a countable ring, put $J=J(R)$, and let $f\in R$ be a nonzero idempotent. Then
\[
 \bigl|(fR^\omega+JR^\omega)/JR^\omega\bigr|=\mathfrak{c}.
\]
Here $JR^\omega$ means $J\cdot R^\omega$, the submodule generated by products $av$ with $a\in J$ and $v\in R^\omega$.
\end{lemma}

\begin{proof}
Every coordinate of an element of $J\cdot R^\omega$ lies in $J$, so $J\cdot R^\omega\subseteq\prod_{n<\omega}J$. Also $f\notin J$: otherwise $1-f$ would be a unit, whereas $f(1-f)=0$ and $f\neq0$.

\smallskip\noindent
For $X\subseteq\omega$, let $v_X\in R^\omega$ have coordinate $f$ on $X$ and $0$ off $X$. Then $fv_X=v_X$, so $v_X\in fR^\omega$. If $X\neq Y$, some coordinate of $v_X-v_Y$ is $f$ or $-f$, and hence does not belong to $J$. Thus $v_X-v_Y\notin JR^\omega$. The cosets $v_X+JR^\omega$ are therefore pairwise distinct, giving the lower bound $\mathfrak{c}$. Countability of $R$ gives the upper bound $|R^\omega|=\mathfrak{c}$.
\end{proof}

We use the following consequence of O'Neill's criterion, in the countable-power case.

\begin{fact}\label{oneill_freeness}
Let $R$ be a countable right coherent ring such that ${_R}R$ is $\Sigma$-algebraically compact. If
$|(fR^\omega+J(R)R^\omega)/J(R)R^\omega|=\mathfrak{c}$ for every primitive idempotent $f\in R$, then $R^\omega$ is free.
\end{fact}

\begin{proof}
The $\Sigma$-algebraic compactness of ${}_{R}R$ implies that $R$ is \emph{semiprimary}: $R/J(R)$ is semisimple artinian and $J(R)$ is nilpotent~\cite[Theorem~14.23]{prest_first_book}. O'Neill's criterion~\cite[Theorem~5.2]{oneill} now applies, since $R$ is right coherent and the assumed cardinal quotients all have size $\mathfrak{c}=|R|^\omega$.
\end{proof}

Recall that $R$ is an $F$-ring if $R^\omega$ is free, and a $P$-ring if $R^\omega$ is projective.

\begin{corollary}\label{F_theorem}
For a countable ring $R$, the following are equivalent: $R$ is an $F$-ring; $R$ is a $P$-ring; and $R$ is left perfect and right coherent. In this case $R^\omega\cong({}_{R}R)^{(\mathfrak{c})}$.
\end{corollary}

\begin{proof}
Every $F$-ring is a $P$-ring. Suppose that $R$ is a $P$-ring and give $P=R^\omega$ its product Polish topology. By Kaplansky's theorem~\cite{Kap}, $P$ has a decomposition into countably generated, hence countable, summands. Theorem~\ref{first_theorem} gives a $\Sigma$-algebraically compact summand $U$ with $\Th(U)=\Th_{\Pun}(P)$. By Proposition~\ref{discrete_product}, $\Th_{\Pun}(P)=\Th(P)$, so $P$, and then its direct summand ${}_{R}R$, are $\Sigma$-algebraically compact. As in Proposition~\ref{useful_prop}, Bass's Theorem gives left perfectness.

\smallskip\noindent
To prove right coherence, fix a pp-formula $\grf(u)$ in one variable and enumerate $\grf({}_{R}R)=\{a_i:i<\omega\}$, allowing repetitions. Then $a=(a_i)_{i<\omega}\in\grf(P)$. Since $P$ is projective, choose homomorphisms $j\colon P\to F$ and $p\colon F\to P$ with $pj=1_P$, where $F$ is free with basis $(e_\lambda)_{\lambda\in\Lambda}$. Write $j(a)=\sum_{k=1}^t b_ke_{\lambda_k}$. The element $j(a)$ belongs to $\grf(F)$, so the direct-sum identity gives $b_k\in\grf({}_{R}R)$ for every $k$. Applying $p$ yields $a=\sum_{k=1}^t b_kp(e_{\lambda_k})$; in each coordinate this says $a_i\in\sum_{k=1}^t b_kR$. Hence $\grf({}_{R}R)=\sum_{k=1}^t b_kR$, since $\grf({}_{R}R)$ is a right ideal. Fact~\ref{coherence_pp} gives right coherence.

\smallskip\noindent
Lemma~\ref{third_condition_pleonastic} and Fact~\ref{oneill_freeness} now show that $P=R^\omega$ is free. Its cardinality is $\mathfrak{c}$, so, as $R$ is countable, its rank is $\mathfrak{c}$. This proves that a countable $P$-ring is an $F$-ring and is left perfect and right coherent. Conversely, Chase's theorem~\cite[Theorem~3.3]{chase} makes every product of projectives projective over a left perfect right coherent ring, in particular $R^\omega$. This proves all the equivalences.
\end{proof}

We can now prove the free-module characterization announced in the introduction.

\begin{restatement}{Theorem}{Free_Polish}
The following are equivalent for a countable ring $R$:
\begin{enumerate}[(1)]
	\item there exists an uncountable free $R$-module with a Polish topology;
	\item the ring $R$ is right coherent and left perfect;
	\item the class $\RProj$ of projective $R$-modules is closed under products; and
	
	\item $({}_{R}R)^{\gro} \isom ({}_{R}R)^{(\mathfrak{c})}$ is a free $R$-module.
	
\end{enumerate}
\end{restatement}

\begin{proof}
Propositions~\ref{useful_prop} and~\ref{useful_prop_2} prove (1)$\Rightarrow$(2). The equivalence (2)$\Leftrightarrow$(3) is Chase's theorem~\cite[Theorem~3.3]{chase}. If (3) holds, $R^\omega$ is projective, and Corollary~\ref{F_theorem} gives (4). Finally, the product of countably many copies of the discrete countable module ${}_{R}R$ is Polish. Transporting that topology along the isomorphism in (4) proves (1).
\end{proof}

\subsection{Uncountable projective Polish modules}\label{sec_projective}

Recall from Subsections~\ref{sec_ziegler} and~\ref{sec_endomorphisms} that an endofinite module has finite length over its endomorphism ring, and is therefore $\Sigma$-algebraically compact. An indecomposable endofinite module is pure-injective with local endomorphism ring. Lemma~\ref{tt_endofinite_point} supplies such a module in every nonempty totally transcendental closed subset of $\Zg(R)$. We first show how its countable power yields an uncountable Polish module.

\begin{lemma}\label{endofinite_power}
Let $R$ be a countable ring. If $V$ is a nonzero indecomposable endofinite $R$-module, then $V$ is countable and $V^\omega\cong V^{(\mathfrak{c})}$.

\end{lemma}

\begin{proof}
By Fact~\ref{fact_closed_point}, $\Cl(V)=\{V\}$, and this singleton is totally transcendental. Lemma~\ref{tt_closed_set} therefore makes $V$ countable. Since pp-formulae commute with products, $V^\omega$ is $\Sigma$-algebraically compact and $\Cl(V^\omega)=\{V\}$. Garavaglia's decomposition theorem~\cite{garavaglia} expresses $V^\omega$ as a direct sum of indecomposable pure-injectives. Every such summand belongs to $\Cl(V^\omega)$, so every summand is isomorphic to $V$. Thus $V^\omega\cong V^{(\lambda)}$ for a cardinal $\lambda$. Because $V$ is nonzero and countable, $|V^\omega|=\mathfrak{c}$, which forces $\lambda=\mathfrak{c}$.
\end{proof}

\begin{lemma}\label{endofinite_projective_idempotent}
Let $R$ be a countable ring. Every nonzero indecomposable endofinite projective $R$-module is isomorphic to $Re$ for a primitive idempotent $e\in R$.
\end{lemma}

\begin{proof}
The idea of the proof follows that of the Krull-Schmidt-Azumaya Theorem~\cite[\S V.5]{stenstrom}.
Let $V$ be such a module and consider an epimorphism $\xymatrix{R^{(I)} \ar@{->>}[r]^p & V}$ from a free module. Because ${_R}V$ is projective, there is a section $s \colon V \to R^{(I)}$, with $ps = 1_V$ so that $V$ is isomorphic to a summand of $R^{(I)}.$ We will use the fact that $S=\End_R \,V$ is a local ring to show that $V$ is in fact a summand of $R^{(J)}$ for some finite $J \subseteq I.$ The lemma then follows from~\cite[Lemma~V.5.2]{stenstrom}, which implies that $V$ is isomorphic to a summand of ${_R}R.$

\smallskip \noindent 
Let $a \in V$ be nonzero and take a finite $J \subseteq I$ with $s(a) \in R^{(J)}.$ Then the composition 
$$\xymatrix{f \colon V \ar[r]^s & R^{(I)} \ar[r]^{\pi} & R^{(J)} \ar[r]^-{p'} & V}$$ is an endomorphism $f \in S$ that fixes $a.$ Here $\pi \colon R^{(I)} \to R^{(J)}$ is the structural projection and $p' \colon R^{(J)} \to V$ is the restriction of $p.$ Because $(1-f)(a) = 0,$ we see that $1-f \in S$ is not an invertible endomorphism of $V.$ But then $f \in S$ must be invertible and $p' \colon R^{(J)} \to V$ is an epimorphism, as required.

\end{proof}

We can now prove the announced characterization of countable rings admitting an uncountable projective Polish module. Recall that a ring is \emph{right artinian} if its right ideals satisfy the DCC. The Hopkins--Levitzki theorem~\cite[Theorem~4.15]{Lam} says that a right artinian ring is right noetherian. Thus its right regular module has both the ACC and the DCC, equivalently finite length; conversely, finite length implies the DCC.

These considerations imply that every countable right artinian ring $R$ is a left $F$-ring. Indeed, every finitely generated right $R$-module $M$ is a quotient of a finite-length module $(R_R)^n$. The kernel has finite length and is therefore finitely generated, so $M$ is finitely presented. In particular, $R$ is right coherent. The DCC on right ideals implies the DCC on principal right ideals, so Bass's theorem (Fact~\ref{bass_perfect}) makes $R$ left perfect. Since $R$ is countable, Corollary~\ref{F_theorem} now gives $R^\omega\cong R^{(\mathfrak c)}$, proving the claim.

\begin{restatement}{Theorem}{Polish projective existence}
The following are equivalent for a countable ring $R$:
	\begin{enumerate}[(1)]
		\item there exists an uncountable Polish projective $R$-module;
		\item there is an irreducible idempotent $e \in R$ such that the projective module $Re$ is of finite length as a right $eRe$-module, with the action given by multiplication;
		\item there is an irreducible idempotent $e \in R$ such that the corner ring $eRe$ is right artinian and the right $eRe$-module $Re$ is of finite length; and
		\item there is a nonzero idempotent $e \in R$ such that the corner ring $eRe$ is a left $F$-ring and $Re$ is a finitely presented right $eRe$-module.
	\end{enumerate}
\end{restatement}

\begin{proof}
Suppose (1) holds, and let $A$ be an uncountable projective Polish $R$-module. Kaplansky's theorem gives a decomposition of $A$ into countable summands. Apply Theorem~\ref{main_th_sec} to obtain $A=B\oplus U$ and the nonempty totally transcendental closed set $\cC=\Cl(\Th_{\Pun}(A))$. Lemma~\ref{tt_endofinite_point} gives an endofinite point $V\in\cC$. By the Main Theorem, $V$ occurs as a direct summand of $U$, and is therefore projective. Lemma~\ref{endofinite_projective_idempotent} yields $V\cong Re$ for a primitive idempotent $e\in R$.

\smallskip\noindent
With our convention on composition, $\End_R(Re)$ is identified with $eRe$ by sending $s\in eRe$ to right multiplication by $s$. Thus endofiniteness of $V$ says precisely that $Re$ has finite length as a right $eRe$-module. This proves (2).

\smallskip\noindent
Now assume (2) with the aim of proving (3) and decompose the right $eRe$-module as a direct sum $Re=eRe\oplus(1-e)Re$. Thus $Re$ has finite length if and only if both summands do. By Hopkins--Levitzki~\cite[Theorem~4.15]{Lam}, the first summand has finite length if and only if the ring $eRe$ is right artinian, which proves (3). In these conditions, $Re$ is also pure-injective, so its endomorphism ring $eRe$ is local.

\smallskip \noindent
To prove that (3) implies (4), use the fact that the corner ring $eRe$ is countable and right artinian, so the paragraph preceding the theorem makes it a left $F$-ring. The right $eRe$-module $Re$ has finite length, so it is finitely generated and hence finitely presented by the same argument. This proves (4). 

\smallskip \noindent Finally, to see that (4) implies (1), use the isomorphism $(eRe)^{\gro} \isom (eRe)^{({\frak c})}$ provided by Corollary~\ref{F_theorem}, since $eRe$ is a countable left $F$-ring, and let us apply the functor $Re \tensor_{eRe} -.$ Because $Re$ is a finitely presented right $eRe$-module, the functor $Re \tensor_{eRe} -$ is finitely presented~\cite[Theorem 10.2.36]{prest_second_book} and therefore respects direct products~\cite[Proposition 18.1.17]{prest_second_book}. We also use the isomorphism $Re\tensor_{eRe}eRe\to Re$ given by $x\tensor a\mapsto xa$, with inverse $x\mapsto x\tensor e$. We obtain
\begin{eqnarray*}
	(Re)^{\gro} \isom (Re \tensor_{eRe}\, eRe)^{\gro} & \isom & Re \tensor_{eRe}\, (eRe)^{\gro} \isom Re \tensor_{eRe} (eRe)^{({\frak c})} \\ 
	& \isom & (Re \tensor_{eRe} eRe)^{({\frak c})} \isom (Re)^{({\frak c})},
\end{eqnarray*}
the second-to-last isomorphism coming from the general property that tensoring respects direct sums. Since $Re$ is nonzero and countable, $(Re)^{\gro}$ with its product topology is an uncountable Polish module. The displayed isomorphism makes it projective, since $Re$ is a direct summand of ${}_R R$. 
\end{proof}

\begin{remark}
For a countable ring $R$, the existence of a closed point in $\Cl({}_{R}R)$ alone does not characterize the rings in Theorem~\ref{Polish projective existence}. For example, $\mathbb Q$ is a closed endofinite point in $\Cl({}_{\mathbb Z}\mathbb Z)$: for every pp-pair $\grf/\psi$, localization carries the equality $\grf(\mathbb Z)=\psi(\mathbb Z)$ to $\grf(\mathbb Q)=\psi(\mathbb Q)$. However, every projective $\mathbb Z$-module is free, and Theorem~\ref{Free_Polish} excludes uncountable projective Polish $\mathbb Z$-modules. In the proof above, projectivity follows from occurrence as an actual summand of $U$.
\end{remark}
\medskip

\begin{example}~\cite{pmcohn}
	P.M.\ Cohn, in his solution to Artin's Problem, constructed an extension of countable division rings $D \leq \grD$ so that the right vector space $\grD_D$ is two-dimensional, while the left vector space ${_D}\grD$ is infinite-dimensional. Consider the matrix ring 
	$$R =  \left\{ \left(\begin{array}{cc} d & \delta \\ 0 & d \end{array} \right) \; \colon \;\; d \in D, \; \grd \in \grD \right \}.$$ 
	This is a local ring, whose maximal ideal is its Jacobson radical 
	${\dis J = \left(\begin{array}{cc} 0 & \Delta \\ 0 & 0 \end{array} \right)},$ $R/J \isom D.$
	
	\smallskip \noindent Because $J^2 = 0,$ both ${_R}R$ and $R_R$ are $\Sigma$-algebraically compact~\cite[Proposition 1.1.6]{prest_second_book}, so that $R$ is both left and right perfect. Let us show that $R$ is right coherent but not left coherent, from which it will follow that it is an example of a left $F$-ring that is not a right $F$-ring. This means that even though the objects appearing in the isomorphism $R^{\gro} \isom R^{({\frak c})}$ of left $R$-modules are both $R$-$R$-bimodules, the isomorphism is decidedly only {\em left} $R$-linear.
	
	\smallskip \noindent Indeed, the right ideals of $J = \grD$ correspond to its right $D$-subspaces, so that the length of $J_R$ is $2.$ This implies that the length of $R_R$ is 3, and that the ring $R$ is right artinian, and therefore right coherent. Because ${_R}R$ is indecomposable and endofinite, it is a closed point so that $\Cl ({_R}R) = \{ {_R}R \}$ is a singleton. The left ideals of $J,$ on the other hand, correspond to the left $D$-subspaces of $\grD,$ from which it follows that ${_R}J$ is not finitely generated, and that the unique simple left $R$-module ${_R}D = R/J$ is not finitely presented. As any simple left ideal of $R$ is isomorphic to ${_R}D,$ it is finitely generated but not finitely presented, and $R$ is not left coherent. 
	
	\smallskip \noindent Another way of seeing that $R$ is not left coherent is by observing that $J = \gra (R_R)$ is pp-definable by the quantifier-free formula $\gra (u) = ua \doteq 0,$ $a \in J$ nonzero, which is not equivalent, relative to $\Th (R_R),$ to any divisibility formula $I|u$ in the language of right $R$-modules: such formulae can only define in $R_R$ a finitely generated left ideal of $J,$ corresponding to a finite dimensional left $D$-subspace of $\grD.$ The pp-type 
	$$\{ \gra (u) \} \cup \{ \neg \, (I|u) \; \colon \;\; I \leq {_D}\grD \; \mrm{finite}{\rm -}\mrm{dimensional} \}$$ is realized by any nonzero element of the simple right $R$-module $D_R \isom R/J,$ which serves as its pure-injective hull. It follows that the totally transcendental closed set $\Cl (R_R) = \{ R_R, D_R \}$ has as its endofinite point $D_R$ and that there is a Kaplansky decomposition of the uncountable Polish module
	$(R_R)^{\gro} \isom (R_R)^{({\frak c})} \oplus (D_R)^{({\frak c})}.$
	Furthermore, the right module $D_R \in \Cl (R_R)$ is not flat, for if it were, then~\cite[Theorem 2.3.9]{prest_second_book} would imply
	$D_R = \gra (D_R) = D_R \gra (R_R) = D_R J = 0.$
\end{example}

\subsection{Existence of Kaplansky decompositions}\label{sec_kaplansky_existence}

The endolength material in Subsection~\ref{sec_endomorphisms} now gives the general existence criterion: Lemma~\ref{tt_endofinite_point} extracts an endofinite point from the closed set supplied by the Main Theorem, and Lemma~\ref{endofinite_power} constructs the required Polish module from that point.

\begin{restatement}{Theorem}{Polish existence}
Let $R$ be a countable ring. There exists an uncountable Polish left $R$-module with a Kaplansky decomposition if and only if there exists a nonzero indecomposable endofinite left $R$-module. Equivalently, if and only if there exists an indecomposable endofinite right $R$-module.
\end{restatement}

\begin{proof}
Suppose that an uncountable Polish $R$-module $A$ has a Kaplansky decomposition. Its unlimited summand gives a nonempty totally transcendental closed set $\cC$ by Theorem~\ref{main_th_sec}. Lemma~\ref{tt_endofinite_point} then supplies a nonzero indecomposable endofinite $R$-module.

\smallskip\noindent
Conversely, let $V$ be such an endofinite module. By Lemma~\ref{endofinite_power}, it is countable and $V^\omega\cong V^{(\mathfrak{c})}$. Thus $V^\omega$, equipped with its product Polish topology, is uncountable and admits a Kaplansky decomposition. By Proposition~\ref{endofinite_left_right}, $\Zg(R)$ contains an endofinite point if and only if $\Zg(R^{\op})$ does. Thus existence of an indecomposable endofinite $R$-module is equivalent to existence of an indecomposable endofinite right $R$-module, proving the last equivalence in the statement. Applying the construction above to the countable ring $R^{\op}$ also gives an uncountable Polish right $R$-module with a Kaplansky decomposition.
\end{proof}

Recall that the unlimited part $T_u$ of a complete theory $T$ of left $R$-modules was defined in~\eqref{def_unlimited_theory} in the introduction; see also~\cite[\S~4.5]{prest_first_book}. If $A$ is a Polish module and $T=\Th(A)$, then
\[
\Cl(\Th_{\Pun}(A))\subseteq\Cl(T_u)\subseteq\Cl(T)=\Cl(A).
\]
Indeed, $T_u$ imposes $\grf=\psi$ for every pp-pair whose index is finite in $T$, while $\Th_{\Pun}(A)$ imposes this equality whenever the index in $A$ is countable, by~\eqref{def_polish_unlimited_theory}. Thus every pp-equality imposed by $T_u$ is imposed by $\Th_{\Pun}(A)$, giving the first inclusion. Every pp-equality imposed by $T$ is imposed by $T_u$, giving the second. 

\begin{restatement}{Corollary}{theory_polish_existence}
	Let $T$ be a complete first-order theory of modules over a countable ring $R$. Then $T$ has an uncountable Polish model admitting a Kaplansky decomposition if and only if there exists a nonzero indecomposable endofinite $R$-module $V \in \Cl (T_u) \subseteq \Zg (R).$ In that case, $B\oplus V^\omega \models T$ is such a Polish model for every countable $B\models T$.
\end{restatement}

\begin{proof}
Let $A\models T$ be an uncountable Polish module with a Kaplansky decomposition. By the Main Theorem, $A = B \oplus M_{\cC}^{(\mathfrak{c})}$ for the nonempty totally transcendental closed set $\cC=\Cl(\Th_{\Pun}(A))$. Lemma~\ref{tt_endofinite_point} supplies an endofinite point 
$V \in \Cl(\Th_{\Pun}(A)) \subseteq \Cl(T_u).$

\smallskip \noindent Conversely, let $V \in \Cl(T_u)$ be endofinite and choose a countable model $B\models T$. By Lemma~\ref{endofinite_power}, $V$ is countable and $V^\omega\cong V^{(\mathfrak{c})}$. Hence $A=B\oplus V^\omega$ is an uncountable Polish module with a Kaplansky decomposition. For every pp-pair $F$, we have $F(A)\cong F(B)\oplus F(V)^{\omega}$. If $F(B)$ is finite, then $V \in \Cl(T_u)$ gives $F(V)=0$, so $F(A)$ has the same finite order. If $F(B)$ is infinite, so is $F(A)$. The Baur--Monk theorem therefore gives $A\equiv B$, and hence $A\models T$.
\end{proof}

Corollary~\ref{theory_polish_existence} and Proposition~\ref{endofinite_left_right} allow us to make more precise the left-right symmetry of the existence of uncountable Polish modules with a Kaplansky decomposition. Recall that if $T$ is a complete theory of left $R$-modules, then the dual theory $T'$ of right $R$-modules is axiomatized by~\cite[Theorem 6.6]{Herzog_duality}
$$T' \models [D\psi \; \colon D\grf] \geq n \;\; \mbox{if and only if} \;\;
T \models [\grf \; \colon \psi] \geq n,$$
where $\grf \mapsto D \grf$ denotes the Prest dual of a pp-formula. It follows from this that the respective unlimited parts are dual: $(T_u)' = (T')_u.$

\begin{corollary}
	A complete first-order theory $T$ of left $R$-modules over a countable ring $R$ has an uncountable Polish model admitting a Kaplansky decomposition if and only if the complete dual theory $T'$ of right $R$-modules has an uncountable Polish model with a Kaplansky decomposition.  
\end{corollary}

\begin{proof}
	By Corollary~\ref{theory_polish_existence}, the theory $T$ has an uncountable Polish module with a Kaplansky decomposition if and only if $\Cl (T_u) \subseteq \Zg (R)$ has an endofinite point if and only if, by Proposition~\ref{endofinite_left_right}, $\Cl ((T_u)') = \Cl ((T')_u) \subseteq \Zg (R^{\op})$ has an endofinite point if and only if, by Corollary~\ref{theory_polish_existence} for $R^{\op},$ the theory $T'$ of right $R$-modules has an uncountable Polish model with a Kaplansky decomposition. 
\end{proof}

\end{document}